\documentclass[]{amsart}

\usepackage{amssymb}
\usepackage{amsthm}
\usepackage{thmtools}
\usepackage{graphicx}
\usepackage{bm}
\usepackage{microtype}
\usepackage{dsfont}
\usepackage{mathtools}
\usepackage{tikz}
\usepackage[backend=biber,bibencoding=utf-8]{biblatex}
\usepackage{comment}
\usepackage{hyperref}
\usepackage[capitalize]{cleveref}
\usepackage{booktabs,multirow}
\usepackage{subcaption}
\usepackage[cachedir=.]{minted}
\usepackage{algorithm}
\usepackage{algpseudocode}
\usepackage{etoolbox}
\usepackage{csquotes}
\hypersetup{hidelinks}
\usepackage{orcidlink}

\newcommand{\pluseq}{\mathrel{+}=}
\newcommand{\asteq}{\mathrel{*}=}

\DeclareMathOperator{\sgn}{sign}
\DeclareMathOperator{\diag}{diag}
\DeclareMathOperator{\Vol}{Vol}

\newcommand{\new}[1]{#1} 

\newcommand{\zeps}[3]{\new{Z_{#1} (#2, #3)}}
\newcommand{\zepsr}[3]{\new{Z_{#1}^{\mathrm{reg}}(#2, #3)}}
\renewcommand{\d}{\,\mathrm{d}}
\newcommand{\sums}{\sideset{}{'} \sum}

\renewcommand{\Re}[1]{\mathrm{Re}(#1)}
\renewcommand{\Im}[1]{\mathrm{Im}(#1)}

\newcommand{\crandallreqs}{
Let \(\Lambda\) be a \(d\)-dimensional lattice,
let \(\bm x,\bm y\in\mathds{R}^d\), and let \(\nu\in\mathds C\)
so that \(\nu\neq d\) if \(\bm y\in\Lambda^{\ast}\). }

\newcounter{mythm}
\crefalias{mythm}{section}
\numberwithin{mythm}{section}

\declaretheorem[style=plain, sibling=mythm]{theorem}
\declaretheorem[style=plain, sibling=mythm]{lemma}

\declaretheorem[style=plain, sibling=mythm]{corollary}
\declaretheorem[style=definition, sibling=mythm]{definition}

\declaretheorem[style=remark, sibling=mythm]{remark}

\graphicspath{{Fig/}}

\newcommand{\algskip}{\vspace*{1.2ex}}

\begin{document}

\title[Computation and properties of the Epstein zeta function]{Computation and Properties of the Epstein Zeta Function with Applications to Quantum Systems}

\author[Buchheit]{Andreas A. Buchheit\textsuperscript{1,2}\orcidlink{0000-0003-4004-713X}}
\author[Busse]{Jonathan K. Busse\textsuperscript{1,3}\orcidlink{0009-0001-3323-3455}}
\author[Gutendorf]{Ruben Gutendorf\textsuperscript{1}\orcidlink{0009-0005-0253-5433}}

\thanks{\textsuperscript{1}Department of Mathematics, Saarland University, Campus E1.1, Saarbrücken, 66123, Saarland, Germany.}
\thanks{\textsuperscript{2}Department of Mathematics, ETH Zürich, Rämistrasse 101, Zürich, 8092, Switzerland.}
\thanks{\textsuperscript{3}Institute of Software Technology, High-Performance Computing Department, German Aerospace Center (DLR), Linder Höhe, Cologne, 51147, North Rhine-Westphalia, Germany.}

\begin{abstract}
The Epstein zeta function generalises the classical Riemann zeta function to oscillatory lattice sums in higher dimensions and has recently emerged as a key tool in the simulation of long-range interacting classical and quantum many-body systems. Its computation and analytic properties are therefore of significant interest, yet a rigorous and comprehensive treatment has been lacking.
We address this gap by introducing a superexponentially convergent algorithm, complete with error bounds, for computing the Epstein zeta function in any dimension with arbitrary real parameters. Our approach is accompanied by a detailed analysis of the analytic properties of the Epstein zeta function. \new{We first present a concise reformulation of its meromorphic continuation, functional equation, and symmetries. We then establish, for the first time, its joint holomorphic continuation in all parameters and offer a complete characterization of the resulting complex singularity structure, which governs convergence rates in numerical algorithms based on the function.} Recognizing that the function can be decomposed into \new{power-law singularities and a regularised analytic part, we provide an algorithm for removing singularities without cancellation error}. This facilitates the evaluation of integrals involving the Epstein zeta function and enables fast precomputations through interpolation methods. \new{It also enables the robust treatment of general Wood-type anomalies in wave scattering problems while avoiding catastrophic cancellation. Finally, it serves as the foundation for a new algorithm for efficiently computing magnetic interactions between arrays of solid bodies.} We present the first high-performance implementation for arbitrary real arguments in EpsteinLib, a C library with Python \new{and Julia} bindings, and rigorously benchmark its performance and accuracy, achieving full-precision evaluation against known analytic results \new{in dimensions 1,2,3,4,6, and 8} and \new{against an arbitrary precision implementation} across the entire parameter range. Finally, we apply our methods to the computation of quantum dispersion relations in \new{three-dimensional} spin systems and Casimir energies in \new{three-dimensional} geometries.
\end{abstract}

\maketitle

\section{Introduction}
Introduced by Paul Epstein in 1903 \cite{epstein1903theorieI,epstein1903theorieII}, the Epstein zeta function generalises the Riemann zeta function to oscillatory lattice sums in higher dimensions while still obeying a functional equation.
For a lattice 
\(\Lambda=A\mathds Z^d\), 
\(A\in\mathds R^{d\times d}\) regular, and
\(\bm{x},\bm{y}\in\mathds{R}^d\),
the Epstein zeta function
is defined as the meromorphic continuation of
\[\zeps{\Lambda, \nu}{\bm x}{\bm y}
	=\sums_{\bm{z}\in\Lambda} \frac{e^{-2\pi i \bm{y}\cdot \bm{z}}}{|\bm{z}-\bm{x}|^{\nu}}
	,\qquad\Re\nu > d,
	\]
 to \(\nu\in\mathds{C}\), where the primed sum excludes the summand \(\bm{z}=\bm{x}\). 
 In private notes, the essential properties of the Epstein zeta functions have already been studied by Hurwitz in 1889 \cite{oswald2016aspects}.
The Epstein zeta function has wide-ranging applications in pure and applied mathematics.
It generalises several special functions, such as the Dirichlet zeta and eta functions, as well as the Lerch zeta function. The non-trivial zeros of the Epstein zeta function have been extensively studied \cite{stark1967zeros,ki2005all}. The Epstein zeta function has recently been used to derive rigorous error bounds in boundary integral equations \cite{wu2021corrected}. It forms a key ingredient of the Singular Euler--Maclaurin expansion (SEM), a recent generalisation of the Euler--Maclaurin summation formula to singular functions in multi-dimensional lattices \cite{buchheit2022efficient,buchheit2022singular}. 

The Epstein zeta function offers numerous applications in theoretical quantum physics and chemistry of technological relevance, especially in systems with long-range interactions. It describes lattice sums, such as Madelung constants, arising in theoretical chemistry \cite{schwerdtfeger100YearsLennardJones2024,emersleben1923zetafunktionenI,emersleben1923zetafunktionenII}. It has been used in the prediction of new phases in unconventional superconductors with long-range interactions \cite{buchheit2023exact}.
We have recently shown that integrals over products of the Epstein zeta function
allow for the evaluation of high-dimensional many-body lattice sums
\cite[Theorem 2.6]{buchheitEpsteinZetaMethod2025}.
The resulting method forms the foundation for a rigorous investigation of the influence of many-body interactions on the stability of matter in theoretical chemistry \cite{robles-navarroExactLatticeSummations2025}.
The Epstein zeta function appears in high-energy physics and quantum field theory in the computation of operator traces relevant in the calculation of Casimir forces
\cite{ambjornPropertiesVacuumMechanical1983}, and in the zeta regularisation of divergent path integrals \cite{hawking1977zeta}. An efficiently computable generalisation of the Epstein zeta function to lattices with boundaries has recently been studied by one of the authors with direct applications to spin systems \cite{buchheitComputationLatticeSums2024}.

The numerical evaluation of the Epstein zeta function and its meromorphic continuation has been the topic of numerous studies. The first meromorphic continuation was constructed by Epstein using theta functions \cite{epstein1903theorieI}. A major milestone in the computation of the Epstein zeta function was reached in the Chowla--Selberg formula \cite{chowla1949epstein}, which allowed for the rapid evaluation of a particular Epstein zeta function in two dimensions. Terras derived an expansion of higher-dimensional zeta functions in terms of zeta functions of lower dimension \cite{terras}. Shanks provided a rapidly convergent representation of Epstein zeta functions for integer arguments \cite{shanksCalculationApplicationsEpstein1975}. A generalisation of the Chowla--Selberg formula to higher dimensions was provided in \cite{elizalde1998multidimensional}. These works rely on the computation of modified Bessel functions of the second kind and related integrals, with their applicability restricted to special cases of the vectors $\bm x,\bm y$ and the lattice $\Lambda$. The modern approach for evaluating the Epstein zeta function is due to Crandall \cite{crandall2012unified}, whose representation offers superexponential convergence in the complete parameter range and relies exclusively upon the computation of incomplete gamma functions. A generalisation of Crandall's representation to point sets without translational invariance has recently been derived by one of the authors \cite{buchheitComputationLatticeSums2024}.

While Crandall's work forms the basis for an algorithm to compute the Epstein zeta function, no library currently offers an implementation. This is due to numerous numerical challenges in its implementation. These challenges include instabilities in incomplete gamma function implementations for negative arguments, as well as issues with the correct choice of truncation values and the treatment of numerical instabilities around singularities. 
Furthermore, the analytic properties of the Epstein zeta function, such as its joint \new{holomorphic extension} and singularity structure have not been rigorously studied yet.

This work addresses the above challenges. We provide a comprehensive account of the analytic properties of the Epstein zeta function and its meromorphic continuation, both in the exponent $\nu\in\mathds C$ and  
\new{in its vector arguments $\bm x,\bm y$ to subsets of $\mathds C^d$}.
Furthermore, we introduce the regularisation of the Epstein zeta function by decomposing it into a power-law singularity and a holomorphic function in a complex neighbourhood of the elementary cell of the reciprocal lattice.
We provide a Crandall-like representation for the Epstein zeta function and its regularisation, along with compact proofs.
Using this representation, we offer the first rigorous discussion of the symmetries, such as the functional equation, singularities, and \new{joint holomorphic extension} of the Epstein zeta function in all its arguments. 
 Based on Crandall's representation, we create an algorithm for computing the Epstein zeta
function and its regularisation in any dimension, for any lattice, and for any choice of
real arguments. We further include an efficient algorithm and implementation of the incomplete
gamma function, which offers full precision even for negative first arguments.
This algorithm is implemented in \href{https://github.com/epsteinlib/epsteinlib}{EpsteinLib}, a first-of-its-kind high-performance C
library for computing the Epstein zeta, including \new{bindings to Python, Julia, and Mathematica}. We benchmark our method using numerous known formulas, obtaining
full precision over the complete parameter range.
Finally, we apply our method to
the computation of quantum dispersion relations in spin systems and the evaluation
of Casimir energies in quantum field theory. 
The analysis of the properties of the Epstein zeta function, together with the algorithm and its implementation in EpsteinLib, aims to establish the Epstein zeta function as a standard special function.

This work is intended for a diverse audience with different interests and goals. We have, therefore, structured it as follows.
In Section \ref{sec:crandall}, we define the Epstein zeta function 
and present its elementary properties. We further provide a new proof of the efficiently computable representation, and discuss the joint \new{holomorphic extension} and the singularities of the Epstein zeta function.
Section \ref{sec:algorithm} introduces our numerical algorithm for its precise and efficient computation.
The precision of our library is demonstrated in extensive numerical experiments against known formulas in Section \ref{sec:experiments}. 
Section \ref{sec:application} showcases the application of our library to quantum spin wave dispersion relations and Casimir energies. We present our conclusions and an outlook in Section \ref{sec:outlook}.

\section{Crandall representation and properties of the Epstein zeta function}
\label{sec:crandall}

\subsection{Definition and elementary properties}
We begin our exposition by introducing the concept of a $d$-dimensional lattice with $d$ a positive integer.

\begin{definition}[Lattices]
For \(A\in\mathds{R}^{d\times d}\) regular, we call the set of points \(\Lambda = A\mathds{Z}^{d}\) a lattice. 
We denote by $E_{\Lambda}=A[-1/2,1/2)^d$
the elementary lattice cell
of \(\Lambda\)
of volume \(V_{\Lambda}=|\:\!\!\det A|\).
We further define the reciprocal lattice \(\Lambda^{\ast}=A^{-T}\mathds{Z}^{d}\)
with the
elementary lattice cell $E_{\Lambda}^{\ast}=A^{-T}[-1/2,1/2)^d$
of volume
\(V_{\Lambda^{\ast}}=1/|\:\!\!\det A|\).
\end{definition}
We then define the Epstein zeta function as follows.
\begin{definition}[Epstein zeta function]
\label{epsteindef}
Let $\Lambda$ be a $d$-dimensional lattice, $\bm x, \bm y \in \mathds{R}^d$ and $\nu \in \mathds{C}$.
The Epstein zeta function is defined as the meromorphic continuation of the lattice sum
\[
		\zeps{\Lambda,\nu}{\bm x}{\bm y} = \sums_{\bm z \in \Lambda} \, \frac{e^{-2\pi i \bm y \cdot \bm z}}{\left|\bm x- \bm z\right|^\nu},\qquad \mathrm{Re}(\nu)>d,
\]
 to \(\nu\in\mathds C\).
  \end{definition}

\new{Here, we have modernized the
notation for the vector arguments $\bm x,\bm y$
compared to the original notation introduced by Epstein \cite{epstein1903theorieI,epstein1903theorieII}.
}
The lattice sum converges absolutely for $\mathrm{Re}(\nu)>d$, as the following Lemma shows.
The idea of this proof can be found in \cite[Lemma 3.14]{Remmert}, who accredited the two-dimensional case to Weierstraß and the generalisation to Eisenstein. 
The domain of definition for the meromorphic continuation is analysed in detail in the next section in Theorem \ref{hol}.

\begin{lemma}
\label{lattice-sum-converges}
Let $\Lambda=A\mathds Z^d$ with $A\in \mathds R^{d\times d}$ regular, and $\nu \in \mathds C$ with $\mathrm{Re}(\nu)>d$. The lattice sum
\[\sums_{\bm z \in \Lambda} \frac{1}{\vert \bm z\vert^{\nu}}\]
then converges absolutely.
\end{lemma}
\begin{proof}
First, recall that  $\big||\bm z|^\nu \big| = |\bm z|^{s}$ with $s=\mathrm{Re}(\nu)$.
We now use the inequality 
\[\vert A\bm z\vert \ge  \Vert A^{-1} \Vert^{-1}  \vert \bm z\vert = c \vert \bm z\vert\]
 to bound the general lattice sum by a sum over $\mathds Z^d$, yielding
\[
\sums_{\bm z\in \Lambda}\frac{1}{\vert \bm z\vert^{s}}
\le 
\frac{2^d}{c^{s}}
\sums_{\bm z\in\mathds N_0^d}\frac{1}{\vert \bm z\vert^{s}}
=
\frac{2^d}{c^{s}}\sum_{n=1}^d\binom{d}{n}\sum_{\bm z\in\mathds N_+^n}\frac{1}{\vert \bm z\vert^{s}}
\]
where the last equality is obtained by reordering the summands with respect to the number $n$ of non-zero entries of $\bm z$.
Therefore, we only need to show that 
\[\sum_{\bm z\in\mathds N^n_+}\frac{1}{\vert \bm z\vert^{s}}\]
converges for \(s>d\) and \(1\le n\le d\).
Inserting the inequality of arithmetic and geometric means,
\[
\frac{z_1^2+\dots + z_n^2}{n} 
\geq (z_1^2 z_2^2\cdots z_n^2)^{1/n},
\]
yields
\[
\sum_{\bm z \in \mathds{N}^n_+} 
\frac{1}{\vert \bm z\vert^{s}} 
\leq
n^{-s/2}\sum_{\bm z \in \mathds{N}^n_+} (z_1 z_2\cdots z_d)^{-s/n}
=n^{-s/2}
\zeta(s/n)^n,
\]
where the Riemann zeta function converges as ${s/n\ge s/d >1}$.
\end{proof}

The Epstein zeta function exhibits numerous symmetries in its arguments $\bm{x}$ and $\bm{y}$, as well as with respect to lattice rescaling, which we summarise in the following lemma.

\begin{lemma}[Symmetries] 
\label{lem:syms}
\crandallreqs Then:
     \begin{enumerate}
\item (Inversion symmetry)  Inversion of $\bm x$ equals inversion of $\bm y$,
           \[\zeps{\Lambda, \nu}{-\bm x}{\bm y} = \zeps{\Lambda, \nu}{\bm x}{- \bm y},~ \text{and}~ \zeps{\Lambda, \nu}{-\bm x}{-\bm y} = \zeps{\Lambda, \nu}{\bm x}{\bm y}.\]
     \item (Translation symmetry) The Epstein zeta function is, up to a prefactor, $\Lambda$-periodic in $\bm x$ and $\Lambda^\ast$-periodic in $\bm y$. For $\bm u \in \Lambda$ and $\bm v \in \Lambda^*$, it holds that
     \[\zeps{\Lambda, \nu}{\bm x+\bm u}{\bm y+\bm v} = e^{-2 \pi i  \bm y \cdot \bm u} \zeps{\Lambda, \nu}{\bm x}{\bm y}.\]
     \item (Scale symmetry) For $s \in \mathds R\setminus\{0\}$,
    \[\zeps{\Lambda, \nu}{\bm x}{\bm y} = |s|^\nu \zeps{s \Lambda, \nu}{s \bm x}{\bm y / s}.\]
     \end{enumerate}
\end{lemma}

\begin{proof}
We may restrict our discussion to $\mathrm{Re}(\nu)>d$ by the uniqueness of the analytic continuation, whose existence we prove in Theorem \ref{hol}. There, the Epstein zeta function is defined via an absolutely convergent lattice sum by Lemma \ref{lattice-sum-converges}. 
The first property is a direct consequence of $-\Lambda=\Lambda$. For the second, we write
\[
    \zeps{\Lambda, \nu}{\bm x+\bm u}{\bm y+\bm v} = e^{-2\pi i \bm y\cdot \bm u}\sums_{\bm z \in \Lambda} \, 
    e^{-2\pi i \bm v \cdot \bm z}\frac{e^{-2\pi i \bm y \cdot (\bm z-\bm u)}}{\left|\bm x- (\bm z-\bm u)\right|^\nu}.
\]
The result follows from 
$\Lambda-\bm u = \Lambda$
and
$e^{-2\pi i \bm v \cdot \bm z} =1$ as $\bm v\cdot \bm z\in \mathds Z$ for $\bm z\in \Lambda$ and $\bm v\in \Lambda^\ast$. Finally, scaling symmetry for the Epstein zeta is obtained via
\[|s|^\nu \zeps{s \Lambda, \nu}{s \bm x}{\bm y / s}
= \sums_{\bm z \in \Lambda}\, |s|^\nu \frac{e^{-2 \pi i \bm y/s \cdot s \bm z}}{\left|s \bm z - s \bm x\right|^\nu} 
= \zeps{\Lambda, \nu}{\bm x}{\bm y}.
\qedhere
\]
\end{proof}

These symmetries can be effectively exploited to simplify computations involving generalised lattice sums; see Section \ref{sec:syms}.
Moreover, translational symmetry allows the consideration of sums over multi-atomic lattices,
where multiple particles with specific weights are arranged in the elementary lattice cell and $\Lambda$-periodically repeated in space.

\begin{remark}[Multi-atomic lattices]
\crandallreqs
Consider $n\in\mathds N_+$ particles at positions $\bm d_1,\ldots,\bm d_n\in E_{\Lambda}$
with weights $g_1,\ldots,g_n\in\mathds R$.
Then the holomorphic continuation of the multi-atomic lattice sum
$$
S=\sum_{i=1}^n \sums_{\bm z\in \Lambda+\bm d_i}g_i\frac{e^{-2\pi i\bm y\cdot \bm z}}{|\bm x-\bm z|^{\nu}}
,\qquad 
\Re\nu >d
$$
in $\nu$ is given in terms of the Epstein zeta function as 
$$
S=\sum_{i=1}^ng_ie^{-2\pi i \bm y\cdot \bm d_i}
\zeps{\Lambda,\nu}{\bm x-\bm d_i}{\bm y}.
$$
\end{remark}

To further discuss the properties of the Epstein zeta function and enable its efficient computation, we derive a representation for its holomorphic continuation in the following section.

\subsection{Computation of the Epstein zeta function and Crandall's representation}

The defining lattice sum of the Epstein zeta function is useful as a numerical tool for $\mathrm{Re}(\nu)\gg d$, yet the sum soon becomes numerically intractable as the real part of $\nu$ approaches $d$. It furthermore
does not allow for insights into the holomorphic continuation in $\nu$ as well as into
the properties of the function in its various parameters.
In this section, we derive
a representation that gives direct access to the holomorphic continuation of the
Epstein zeta function in all parameters. The representation is based on works
by Crandall \cite{crandall2012unified} with a compact reformulation in \cite{buchheitComputationLatticeSums2024}. It serves both as the basis for our numerical algorithm
and as a tool for proving the analytical properties of the Epstein zeta function.

We start by defining the Fourier transform of integrable functions.
\begin{definition}[Fourier transform]
Let $f :\mathds{R}^d \rightarrow \mathds{C}$ be integrable. We define its Fourier transform $\mathcal{F} f = \hat{f}$ via
\[
(\mathcal{F}f)(\bm k)
=\int_{\mathds{R}^d} f(\bm x)e^{-2\pi i\bm k\cdot\bm x}\,\mathrm{d}\bm x
,\qquad \bm k \in\mathds R^d.
\]
\end{definition}

The Poisson summation formula allows us to cast oscillatory sums over a lattice $\Lambda$ in terms of sums of the Fourier transform of the summand function over the reciprocal lattice $\Lambda^*$. 
\begin{lemma}[Poisson summation formula]
Let $f:\mathds{R}^d \rightarrow \mathds{C}$ be a continuous integrable function. If there exist $C,\varepsilon>0$ such that
\[|f(\bm z)|+|\hat{f}(\bm z)| \leq C (1+ |\bm z|)^{-d-\varepsilon},\quad \bm z\in \mathds R^d,\]
then 
\[V_\Lambda \sum_{\bm z \in {\Lambda}} f(\bm z ) e^{-2\pi i \bm z \cdot \bm y} = \sum_{\bm k \in \Lambda^\ast} \hat{f}(\bm k+\bm y),\quad \bm y\in \mathds R^d.\]
\end{lemma}
For the lattice $\Lambda=\mathds Z^d$, the result is well-known, see \cite[Corollary 2.6]{steinweiss1972fourier}. The case of general lattices is a slight generalisation, shown for instance in \cite[Lemma 3.2]{buchheit2022singular}.

The incomplete gamma functions and their regularisation form key ingredients in the efficiently computable representation of the Epstein zeta function.

\begin{remark}
\label{gammaGammaStar}
We denote by $\Gamma(s)$ the holomorphic extension of the gamma function 
$$
\new{
\Gamma(s)=
\int_0^{\infty}t^{s}e^{-t}\frac{\d t}{t}
,\qquad \Re{s}>0,
}
$$
to $s\in\mathds C\setminus(-\mathds N_0)$.
For any $x\in\mathds C$, $\Re x > 0$,
we denote 
by $\Gamma(s,x)$ the holomorphic extension of the upper incomplete Gamma function
$$
\new{
\Gamma(s,x)=
\int_x^{\infty}t^{s}e^{-t}\frac{\d t}{t}
,\qquad \Re{s}>0,
}
$$
to  
$s\in\mathds C$ 
and 
we denote by $\gamma(s,x)$ the holomorphic extension of the lower incomplete gamma function 
to 
$s\in \mathds C\setminus (-\mathds N_0)$,
\new{which we define by}
the fundamental relation
$$\Gamma(s)=\Gamma(s,x)+\gamma(s,x),
$$
see \cite[Section §8.2(ii)]{NIST:DLMF}.
Tricomi's twice regularised lower incomplete 
gamma function~$\gamma^*(s,x)$ is the holomorphic extension
$$
\gamma^*(s, x) = \frac{\gamma(s, x)}{x^s \Gamma(s)}
,\qquad s,x>0,
$$
to a jointly entire function in  $(s,x)\in\mathds C\times \mathds C$, see \cite{Tricomi}. 
\end{remark}

The key idea for the Crandall representation is to decompose the function $\vert  \bm \cdot\vert^{-\nu}$ into a superexponentially decaying function that includes the singularity at $\bm 0$ and a smooth function that includes the asymptotic power-law decay. We call these functions Crandall functions.
\new{Closely related decompositions already appear in the classical literature. Riemann's meromorphic continuation of the Riemann zeta function uses a corresponding splitting after a Mellin transform, together with Poisson summation \cite{riemann}. The same general mechanism underlies the meromorphic continuation of the Epstein zeta function via generalized theta functions \cite{epstein1903theorieI,epstein1903theorieII}. In the particular case $\nu=1$ (Coulomb interaction), the Crandall representation reduces to the classical Ewald splitting \cite{ewald1921berechnung}. Its key feature is therefore not the existence of the splitting itself, but its formulation directly at the level of the kernel $\vert \bm\cdot\vert^{-\nu}$. The kernel-level representation in terms of Crandall functions makes the analytic structure clear, enables rigorous subtraction of singular terms without catastrophic cancellation, and provides a natural foundation for an efficient numerical implementation.}

\begin{definition}[Crandall functions]
\label{defCrandallFuncs}
Let $\nu\in\mathds C$ and $\bm z\in\mathds R^d\setminus\{\bm 0\}$.
We define the upper Crandall function $G_{\nu}(\bm z)$ as
$$
G_{\nu}(\bm{z})= 
	\frac{\Gamma(\nu/2,\pi \bm{z}^2)}{(\pi \bm{z}^2)^{\nu/2}}
,\qquad G_{\nu}(\bm 0)
=-\frac{2}{\nu},
$$
\new{using the notation
\[\bm{z}^2=\sum_{j=1}^d z_j^2.
\]}
Let $\nu\in\mathds C\setminus(-2\mathds N_0)$ and $\bm z\in \mathds C^d$. Then, the lower Crandall function $g_{\nu}(\bm z)$ is defined as
$$
g_{\nu}(\bm z)=
\Gamma(\nu/2) \gamma^*(\nu/2,\pi \bm z^2)= \frac{\gamma(\nu/2,\pi \bm{z}^2)}{(\pi \bm{z}^2)^{\nu/2}} .
$$
\end{definition}

In particular, the upper Crandall function $G_{\nu}(\bm z)$ admits superexponential decay in the second argument,
which permits truncation of the resulting lattice sums in the computation of the Epstein zeta function.

The following lemma discusses the \new{holomorphic extension} of the Crandall functions in both arguments, \new{preparing the holomorphic extension of the Epstein zeta function itself,} and provides bounds that describe their asymptotic decay as $\vert \bm z\vert\to \infty$.

\begin{lemma}[Properties of Crandall functions]
\label{lem:propCrandall}
Let $\nu \in\mathds C$ and $\bm z\in D$ with the complex sector
$$
D=\{\bm u\in\mathds C^d:|\Re{\bm u}|>|\Im{\bm u}|\},
$$
where real and imaginary parts are applied component-wise and where $\vert \bm \cdot \vert$ denotes the Euclidean norm.
\begin{enumerate}
\item (Holomorphy) The lower Crandall function $g_{\nu}(\bm z)$ 
\new{can be jointly holomorphically extended to} 
$(\nu,\bm z)\in\mathds C\setminus(-2\mathds N_0)\times \mathds C^d$. The upper Crandall function $G_{\nu}(\bm z)$ can be extended to a jointly holomorphic function in $(\nu,\bm z)\in \mathds C \times D$. Finally,
$G_{\nu}(\bm 0)=-2/\nu$ is holomorphic in $\nu\in\mathds C\setminus\{0\}$
with a simple pole at $\nu=0$.
\item (Fundamental relation)
For any $\lambda>0$, it holds
\[
( \bm z^2 )^{-\nu/2} = \frac{(\pi/\lambda^2 )^{\nu/2}}{\Gamma(\nu/2)} \Big(G_\nu(\bm z/\lambda)+g_\nu(\bm z/\lambda)\Big).
\]
Further, for $\mathrm{Re}(\nu)<0$, we have $\lim_{\alpha \to 0_+} G_\nu(\alpha\bm z)=G_\nu(\bm 0)=-2/\nu$. For all $\nu \in \mathds C$, we have
\[
\frac{1}{\Gamma(\nu/2)}(G_\nu(\bm 0)+g_\nu(\bm 0)) = 0,
\]
\new{where division by the Gamma function removes the poles at $\nu\in -2 \mathds N_0$.}
\item (Uniform bound)
Define
\[
u=\pi \mathrm{Re}(\bm z^2),\quad v=\max\{0,\mathrm{Re}(\nu/2)-1\}.\] 
We then have
$$
|G_\nu(\bm z)| \leq \frac{e^{-u}}{u-v}, \quad u>v.
$$
\end{enumerate}
\end{lemma}
\begin{proof}
    (1) (Holomorphy) The lower Crandall function is the product of $\Gamma(\nu/2)$, which is holomorphic on $\mathds C\setminus (-2\mathds N_0)$, and the jointly entire function $(\nu,\bm z)\mapsto\gamma^*(\nu/2,\pi \bm z^2)$. Thus, \new{it can be holomorphically extended to} $(\nu,\bm z)\in (\mathds C\setminus (-2\mathds N_0))\times \mathds C^d$. For the upper Crandall function, first notice that $f(a,z)=\Gamma(a,z)/z^a$ is jointly holomorphic in $a\in \mathds C$ and $z\in \mathds C$ with $\mathrm{Re}(z)>0$, which follows immediately from the representation \cite[Eq.~(2.1)]{borwein2009uniform},
    \[
    f(a,z)=e^{-z}\int_{0}^\infty e^{-z t}(1+t)^{a-1}\,\mathrm d t.
    \]
    As $\bm z\in D$ is equivalent to $\mathrm{Re}(\bm z^2)>0$, we have that 
    \[
    G_\nu(\bm z) = f(\nu/2, \pi \bm z^2) 
    \]
    is holomorphic on $(\nu,\bm z) \in \mathds C\times D$. 

    (2) (Fundamental relation) The fundamental relation of the Crandall functions follows directly after inserting their definition from the associated relation of the incomplete Gamma functions in Remark \ref{gammaGammaStar}. Furthermore, note that
    \[
    \lim_{\alpha\to 0_+}f(a,\alpha z) = \int_{0}^\infty (1+t)^{a-1}\d t=-1/a,\quad \mathrm{Re}(a)<0.
    \]
    Thus for fixed $\nu$, $G_\nu(\bm z)$ can be continuously extended to $\bm z= \bm 0$ and the limit coincides with the definition $G_\nu(\bm 0)=-2/\nu$. Finally, for $\mathrm{Re}(\nu)<0$, we have
    \[
    \lim_{\alpha \to 0_+} ((\alpha\bm z)^2)^{-\nu/2} = 0 = \frac{1}{\Gamma(\nu/2)} \Big(G_\nu(\bm 0)+g_\nu(\bm 0)\Big).
    \]
    
    (3) (Uniform bound) 
    Using the integral representation above, we find with $w = \mathrm{Re}(u)$ and $t = \mathrm{Re}(a)$
    \[
    \vert f(a,u)\vert \le  e^{-w} \int_0^\infty  e^{-w\mu}(1+\mu)^{t-1}\mathrm d \mu.
    \]
    Following the proof of Theorem 2.1 in \cite{borwein2009uniform}, we use the estimates $(1+\mu)^{t-1}\le(e^\mu)^{t-1}$ for $t\ge 1$ and $(1+\mu)^{t-1}< 1$  for $t<1$. Inserting them into the integral above yields
    \[
    \vert f(a,z)\vert\le  e^{-w} \frac{1}{w-\max\{0,t-1\}},\quad w>\max\{0,t-1\}.
    \]
    We find the desired uniform bound after setting $a=\nu/2$ and $u=\pi \bm z^2$.
\end{proof}

The Fourier transform connects the lower to the upper Crandall function.
\begin{lemma}
\label{gn-fourier}
Let $\mathrm{Re}(\nu)>d$ and $\lambda >0$. Then
\[\mathcal{F}(g_\nu(\cdot/\lambda))(\bm k) = \lambda^d G_{d-\nu}(\lambda \bm k)
,\qquad \bm k\in\mathds R^d.\]
\end{lemma}
\begin{proof}
We write the lower Crandall function in terms
of a     defining integral
\[g_\nu(\bm z / \lambda) = \Bigl(\frac{\lambda^2}{\pi \bm z^2} \Bigr)^{\nu/2} \int \limits_0^{\pi \bm z^2 / \lambda^2} t^{\nu/2} e^{-t}\,\frac{\rm{d}t}{t},\]
see \cite[Eq. (6.5.2)]{abramowitz}.
Substituting $t = \pi \bm z^2 s^2$ gives
\[g_\nu(\bm z / \lambda) = 2 \lambda^\nu \int \limits_0^{1/\lambda} s^{\nu} e^{-\pi \bm z^2 s^2}\, \frac{\rm{d}s}{s}.\]
The Fourier transform now readily follows from 
$\mathcal F e^{-\pi \bm (\cdot)^2 s^2} = s^{-d}e^{-\pi \bm (\cdot)^2/s^2}$ and reads
\[
\mathcal F \big(g_\nu (\bm \cdot/\lambda)\big) (\bm k) = 2 \lambda^\nu \int \limits_0^{1/\lambda} s^{\nu-d} e^{-\pi \bm{k}^2/s^2} \frac{\rm{d}s}{s},\]
where the exchange of the integration order is justified due to the Fubini-Tonelli theorem, as the resulting integrand is absolutely integrable due to the restriction $\mathrm{Re}(\nu)>d$.
For $\bm k = \bm 0$, the integral evaluates to
\[\mathcal{F}(g_\nu)(\bm 0) = -\lambda^d \frac{2}{d-\nu} = \lambda^dG_{d-\nu}(\bm 0).\]
For $\bm k\neq \bm 0$, we obtain the desired result after substituting $t = \pi \bm k^2 / s^2$,
\[\mathcal{F}(g_\nu)(\bm k) = 
\lambda^d (\pi \lambda^2 \bm k^2)^{-(d-\nu)/2} 
\int_{\pi \lambda^2 \bm k^2}^\infty t^{(d-\nu)/2} e^{-t}\,\frac{\rm{d}t}{t} = \lambda^d G_{d-\nu}(\lambda \bm k).
\qedhere
\]
\end{proof}

Crandall's representation now follows by rewriting the lattice sum for the Epstein zeta function in terms of the upper and lower Crandall functions. By applying the Poisson summation formula to the sum involving the lower Crandall functions, the Epstein zeta function is expressed in terms of two superexponentially convergent lattice sums, thus forming the basis for its efficient computation. The resulting representation can be considered as a generalization of Ewald's method  \cite{ewald1921berechnung}.

  \begin{theorem}[Crandall representation]
  \label{crandall}
 \crandallreqs
	 Then  the Epstein zeta function is well-defined and
	 for any $\lambda>0$, it holds that
\begin{align*}
\zeps{\Lambda,\nu}{\bm{x}}{\bm{y}}
=\frac{(\pi/\lambda^2 )^{\nu/2}}{\Gamma(\nu/2)}\Bigg[
\sum_{\bm{z}\in\Lambda}
&G_{\nu}\Big(\frac{\bm{z}-\bm{x}}{\lambda}\Big)e^{-2\pi i\bm{y}\cdot\bm{z}}
\\
&
+\frac{\lambda^d}{V_{\Lambda}}
\sum_{\bm{k}\in\Lambda^{\ast}}G_{d-\nu}
(\lambda(\bm{k}+\bm{y}))e^{-2\pi i\bm{x}\cdot(\bm{k}+\bm{y})}
\Bigg].
\end{align*}
\end{theorem}

\begin{proof}
For the first part, we assume $\nu \in \mathds{R}$ such that $\nu > d$. Then the defining lattice sum of the Epstein zeta function converges absolutely by  
Lemma \ref{lattice-sum-converges}.

We then use the fundamental relation 
 in 
 Lemma \ref{lem:propCrandall} (2) to separate the power-law singularity into the upper and lower Crandall functions
with a Riemann splitting parameter $\lambda>0$.
Inserting the result into the definition of the Epstein zeta function yields
\[\zeps{\Lambda,\nu}{\bm x}{\bm y} \frac{\Gamma(\nu/2)}{(\pi/\lambda^2)^{\nu/2}}
= \sum_{\bm z \in \Lambda\setminus\{\bm x\}} G_\nu\Big(\frac{\bm{z}-\bm{x}}{\lambda}\Big) e^{-2\pi i \bm y \cdot\bm z}+\sum_{\bm z \in \Lambda\setminus\{\bm x\}} g_\nu\Big(\frac{\bm{z}-\bm{x}}{\lambda}\Big) e^{-2\pi i \bm y \cdot\bm z} .
\]
Since \(\big(G_{\nu}(\bm 0)+g_{\nu}(\bm 0)\big)/\Gamma(\nu/2)=0\), we can extend the summation to the whole lattice $\Lambda$ in both sums without altering the result. The first sum already appears in Crandall's representation, whereas the second is brought to the desired form through Poisson summation.  
Note that $g_{\nu}$ is smooth by 
Definition \ref{defCrandallFuncs} and decreases sufficiently fast for any choice of $\nu>d$ as a consequence of the fundamental relation 
Lemma \ref{lem:propCrandall} (2), because for positive real-valued $\nu$, the Crandall functions are both non-negative. By 
Lemma \ref{gn-fourier}, we have that  $\mathcal F g_{\nu}(\bm \cdot/\lambda) = \lambda^d G_{d-\nu}(\lambda \,\bm \cdot)$, which is a continuous and superexponentially decreasing function by 
Lemma \ref{lem:propCrandall}\,(3). Thus, the requirements for 
Poisson summation hold, and the resulting sum reads
\[\frac{\lambda^d}{V_\Lambda} \sum_{\bm k \in \Lambda^*} G_{d-\nu}\bigl(\lambda(\bm k + \bm y)^2 \bigr) e^{-2\pi  i  \bm x \cdot (\bm k + \bm y)},\]
obtaining the desired representation.

In Theorem \ref{hol}, we show that the Crandall representation converges to a holomorphic function in $\nu \in \mathds{C} \setminus \{d\}$ if $\bm y \in \Lambda^*$ and to an entire function in $\nu \in \mathds{C}$ if \(\bm y\in\mathds R^d\setminus\Lambda^*\). 
By the identity theorem of holomorphic functions, this implies that Crandall's representation provides the unique holomorphic continuation of the Epstein zeta function.
\end{proof}

The advantage of the Crandall representation over other ways to compute the
 Epstein zeta function is the superexponential decay of the function $G_{\nu}(\bm \cdot)$. The
 resulting two lattice sums can be truncated to lattice points within a ball of radius
 $r >0$ around $\bm x$ and $\bm y$, as discussed in Section \ref{sec:algorithm}, yielding an efficiently computable
 representation.

\subsection{Analytic properties}
\label{sec:props}

Crandall’s representation forms the basis of our algorithm for the numerical evaluation of the Epstein zeta function and enables a rigorous analysis of its analytic properties.  In particular, the functional equation of the Epstein zeta function follows directly from this representation.

\begin{corollary}[Functional equation]
\label{funqeq}
\crandallreqs
The product
\[
\frac{(V_{\Lambda}^{2/d}/\pi)^{\nu/2}}{\Gamma\big((d-\nu)/2\big)}
e^{\pi i \bm x \cdot \bm y}\zeps{\Lambda, \nu}{\bm x}{\bm y}
\]
is then invariant under the parameter transformation
\[(\Lambda,\nu,\bm x,\bm y)\to 
(\Lambda^{\ast},d-\nu,\bm y,-\bm x).\]
\end{corollary}

\begin{proof}
We start by inserting Crandall's representation for $\lambda = 1$ in the above formula. After exchanging the first and the second sum, we obtain
\begin{align*}
\frac{V_{\Lambda}^{\nu/d}/V_{\Lambda^*}}{\Gamma((d-\nu)/2) \Gamma(\nu/2)} 
e^{\pi i \bm x \cdot \bm y}\sum_{\bm{k}\in\Lambda^{\ast}}G_{d-\nu}
(\bm{k}+\bm{y})e^{-2\pi i\bm{x}\cdot(\bm{k}+\bm{y})}
\\
+\frac{V_{\Lambda}^{\nu/d}}{\Gamma((d-\nu)/2) \Gamma(\nu/2)} 
e^{\pi i \bm x \cdot \bm y}\sum_{\bm{z}\in\Lambda}
G_{\nu}(\bm{z}-\bm{x})e^{-2\pi i\bm{y}\cdot\bm{z}}.
\end{align*}
After noticing that \(V_{\Lambda}=1/V_{\Lambda^{\ast}}\), swapping the indices \(\bm k\) and \(\bm z\)
and using that  \(G_\nu(-\,\bm \cdot)=G_\nu(\bm \cdot)\),
we have
\begin{align*}
\frac{V_{\Lambda^*}^{(d-\nu)/d}}{\Gamma((d-\nu)/2) \Gamma(\nu/2)} 
e^{-\pi i \bm x \cdot \bm y}\sum_{\bm{z}\in \Lambda^{\ast}}G_{d-\nu}
(\bm{z}-\bm{y})e^{-2\pi i\bm{x}\cdot (-\bm{z}+\bm y)}
\\
+\frac{V_{\Lambda^*}^{(d-\nu)/d}/V_{(\Lambda^*)^*}}{\Gamma((d-\nu)/2) \Gamma(\nu/2)} 
e^{-\pi i \bm x \cdot \bm y}\sum_{\bm{k}\in(\Lambda^{\ast})^{\ast}}
G_{\nu}(\bm{k}+(-\bm{x}))e^{-2\pi i\bm{y}\cdot\bm{k}}.
\end{align*}
The above expression is equal to Crandall's representation of the transformed product
\[
\frac{(V_{\Lambda^{\ast}}^{2/d}/\pi)^{(d-\nu)/2}}{\Gamma(\nu/2)}e^{-\pi i \bm x \cdot \bm y}
\zeps{\Lambda^{\ast},d-\nu}{\bm y}{-\bm x}.
\qedhere
\]
\end{proof}

The following theorem shows that the Epstein zeta function can be holomorphically extended not only in $\nu$ but also in $\bm x$ and $\bm y$. \new{To the best of our knowledge, this is the first proof of this result.} \new{The holomorphic extension} has several important consequences. First, it establishes the natural domain of definition for the Epstein zeta function with complex arguments
$\bm x$, $\bm y$ and $\nu$. Since holomorphy implies smoothness, the theorem guarantees the existence of arbitrary higher-order derivatives of the Epstein zeta function with respect to its arguments. Finally, the domain of holomorphy allows for a straightforward determination of the radius of convergence for Taylor series expansions in these arguments. \new{These results are of central importance in applications. In particular, they determine the convergence rate of quadrature rules for integral involving the Epstein zeta function in an algorithm that permits exponential speedups in computations of many-body interactions in chemistry, see \cite{buchheitEpsteinZetaMethod2025,robles-navarroExactLatticeSummations2025}. They also determine the convergence rate of a new algorithm for computing magnetic interactions between arrays of solid bodies based on derivatives of generalized Epstein zeta functions in \cite{buchheitZetaExpansionLongrange2025}.}

\begin{theorem}[\new{Holomorphic extension of the Epstein zeta function}]
\label{hol}
Let $\Lambda$ be a $d$-dimensional lattice and $\nu\in\mathds C$. We define the set $\new{D_L}\subseteq \mathds C^d$ as the  following intersection of $d$-dimensional complex cones with origins at $L\subseteq \mathds R^d$,
$$
D_L=\{\bm u\in\mathds C^d:|\Re{\bm u}-\bm z|>|\Im{\bm u}|\ \forall \bm z\in L\}.
$$
For $\bm x \notin \Lambda$, $\bm y \notin \Lambda^*$, the Epstein zeta function can be holomorphically extended to  
$$
(\nu,\bm x,\bm y)\in 
\mathds C\times D_{\Lambda}\times D_{\Lambda^*}.
$$
For $\bm x\in\Lambda$, $\bm y \notin \Lambda^*$, the Epstein zeta function can be holomorphically extended to 
$$(\nu,\bm y)\in \mathds C\times D_{\Lambda^*}$$
and for
$\bm y \in \Lambda^*$, $\bm x \notin \Lambda$, the Epstein zeta  function can be holomorphically extended to  
$$(\nu,\bm x)\in \Big( \mathds C\setminus\{d\}\Big)\times D_{\Lambda}$$
with a simple pole in $\nu = d$.
Finally, for $\bm x\in\Lambda$ and $\bm y\in\Lambda^*$, the Epstein zeta function is holomorphic in $\nu\in\mathds C\setminus\{d\}$ with a simple pole in $\nu=d$.  
\end{theorem}

\begin{proof}
Consider the set 
$W=
\mathds C\times D_\Lambda \times D_{\Lambda^\ast}$.
By
Lemma \ref{lem:propCrandall},
and Hartogs' theorem \cite[Theorem 2.2.8]{hormanderIntroductionComplexAnalysis1990},
 each summand in Crandall's representation is jointly holomorphic in $(\nu,\bm x,\bm y) \in W$. We now need to show that holomorphy is preserved for the infinite lattice sums.  
For $r>0$, we define $f_r: W\to \mathds C$,
\begin{align*}
f_r(\nu,\bm x,\bm y) =
\sum_{\substack{\bm{z}\in\Lambda \\ |\bm z|\le r}}
&G_{\nu}\Big(\frac{\bm{z}-\bm{x}}{\lambda}\Big)e^{-2\pi i\bm{y}\cdot\bm{z}}
+\frac{\lambda^d}{V_{\Lambda}}
\sum_{\substack{\bm{k}\in\Lambda^\ast \\ |\bm k|\le r}}G_{d-\nu}
(\lambda(\bm{k}+\bm{y}))e^{-2\pi i\bm{x}\cdot(\bm{k}+\bm{y})},
\end{align*}
which is holomorphic since it is the finite sum of holomorphic functions. We now show that $f_r\to f$ compact uniformly as $r \to \infty$ with $f:W\to \mathds C$,
\[
\frac{(\pi/\lambda^2 )^{\nu/2}}{\Gamma(\nu/2)} f(\nu,\bm x,\bm y) = \zeps{\Lambda,\nu}{\bm x}{\bm y}.
\]
To this end, choose $\mathcal K\subset W$ compact.  Using
Lemma \ref{lem:propCrandall} (3), we can derive uniform bounds for the upper Crandall functions as follows
 \[
 \left\vert G_\nu\left(\frac{\bm z-\bm x}{\lambda}\right)\right\vert \le e^{-\pi \mathrm{Re}((\bm z-\bm x)^2/\lambda^2)},\quad \pi \mathrm{Re}((\bm z-\bm x)^2/\lambda^2)>\max\{1,\mathrm{Re}(\nu/2)\}.
 \]
 With \[\mathrm{Re}\Big((\bm z-\bm x)^2\Big)= (\bm z-\mathrm{Re}(\bm x))^2-\mathrm{Im}(\bm x)^2\ge \big(|\bm z|-|\bm x|\big)^2-\vert \bm x\vert^2\ge \vert \bm z/2\vert^2\] 
 for all $\bm z\in \Lambda$ with $\vert \bm z \vert\ge 4 \sup_{\mathcal K}\vert \bm x\vert$, we find
 \[
 \left\vert G_\nu\left(\frac{\bm z-\bm x}{\lambda}\right)\right\vert \le e^{-\pi \vert \bm z/(2\lambda)|^2 },\quad \pi \vert \bm z/(2\lambda)\vert^2> \sup_{\mathcal K} \max\{1,\mathrm{Re}(\nu/2)\}.
 \]
 Similarly
 \[
 \left\vert G_{d-\nu}\left(\lambda(\bm k+\bm y)\right)\right\vert \le e^{-\pi \vert \lambda \bm k/2|^2 },\quad \pi \vert \lambda \bm k/2|^2 > \sup_{\mathcal K} \max\{1,\mathrm{Re}((d-\nu)/2)\},
 \]
 for all $\bm k\in \Lambda^\ast$ with $\vert \bm k \vert\ge 4 \sup_{\mathcal K}\vert \bm y\vert$. Finally, we have
\begin{align*}
&\sup_{\mathcal K}\vert e^{-2\pi i \bm y\cdot \bm z}\vert \le e^{c_1 \vert \bm z\vert},\quad  \sup_{\mathcal K}\vert e^{-2\pi i \bm x\cdot (\bm k+\bm y)}\vert \le    c_2 e^{c_3 \vert \bm k\vert}.
\end{align*}
For $r$ chosen large enough to fulfill the conditions for the above bounds, we find
\[
\sup_{\mathcal K}\vert f-f_r| \le \sum_{\substack{\bm{z}\in\Lambda \\ |\bm z|>r}} e^{-\pi (\bm z/(2\lambda))^2}e^{c_1 \vert \bm z\vert}+\frac{\lambda^d}{V_\Lambda}\sum_{\substack{\bm{k}\in\Lambda^\ast \\ |\bm k|>r}} e^{-\pi (\lambda \bm k/2)^2} c_2 e^{c_3 \vert \bm k\vert} \to 0,
\]
as $r\to \infty$.
Convergence is obtained from 
Lemma \ref{lattice-sum-converges} using that the superexponential decay of the Gaussian dominates the exponential and that the resulting summand falls off faster than any power-law. From compact uniform convergence of $f_r\to f$ then follows holomorphy of $f$ and thus of the Epstein zeta function on $W$, see \cite[Theorem 10.28]{rudin1987real}, which is a direct consequence of Morera's theorem. 

Now consider the case $\bm x\in \Lambda$ and $\bm y \in D_{\Lambda^\ast}$. The proof then proceeds in analogy to before, where only the term  $\bm z=\bm x$ in the first sum requires further attention. Recalling that $G_\nu(\bm 0)=-2/\nu$, we find
$$
\frac{(\pi/\lambda^2 )^{\nu/2}}{\Gamma(\nu/2)} G_\nu(\bm 0) e^{-2\pi i \bm y\cdot \bm z} = -\frac{(\pi/\lambda^2 )^{\nu/2}}{\Gamma(\nu/2+1)}e^{-2\pi i \bm y\cdot \bm z},$$
as $1/\Gamma(s+1)=1/(s\Gamma(s))$, for $s\in \mathds C$.
The apparent singularity at \(\nu=0\) is therefore removable. Thus the Epstein zeta function
\new{can be holomorphically extended to}
in $(\nu,\bm y) \in \mathds C\times D_{\Lambda^\ast}$.

In case that $\bm y\in\Lambda^*$, we find that the term $\bm k = -\bm y$ is of the form
\[
\frac{(\pi/\lambda^2 )^{\nu/2}}{\Gamma(\nu/2)}
\frac{\lambda^d}{V_{\Lambda}}
\frac{2}{(\nu-d)}e^{-2\pi i\bm{x}\cdot(\bm{k}+\bm{y})},
\]
which exhibits a simple pole at $\nu = d$. Thus for $\bm x\in \Lambda$, the Epstein zeta function 
\new{can be holomorphically extended to}
$(\nu, \bm y)\in (\mathds C\setminus \{d\})\times D_{\Lambda^\ast}$. Finally, the discussion of the $\bm z = \bm x$ and $\bm k = -\bm y$ terms above also shows that for $\bm x \in \Lambda$ and $\bm y\in \Lambda^\ast$, the Epstein zeta function is holomorphic in $\nu \in \mathds C\setminus\{d\}$ with a simple pole at $\nu = d$.
\end{proof}

\new{Since the holomorphic extension agrees with the Epstein zeta function on its original domain of definition, its uniqueness follows from the identity theorem. Hence, in the following, the Epstein zeta function and related functions for complex arguments are to be understood in terms of their respective holomorphic extensions.}

\new{
Note that the domain of holomorphy of the Epstein zeta function implies uniform exponential convergence of the Taylor series in $\bm x$ on every ball that does not intersect $\Lambda$, which is used to determine the rate of convergence of our recent algorithm in \cite[Thm.~2.3]{buchheitZetaExpansionLongrange2025}.
}

\subsection{Singularities and regularisation}

\begin{figure}
    \centering
    \includegraphics[width=1\linewidth]{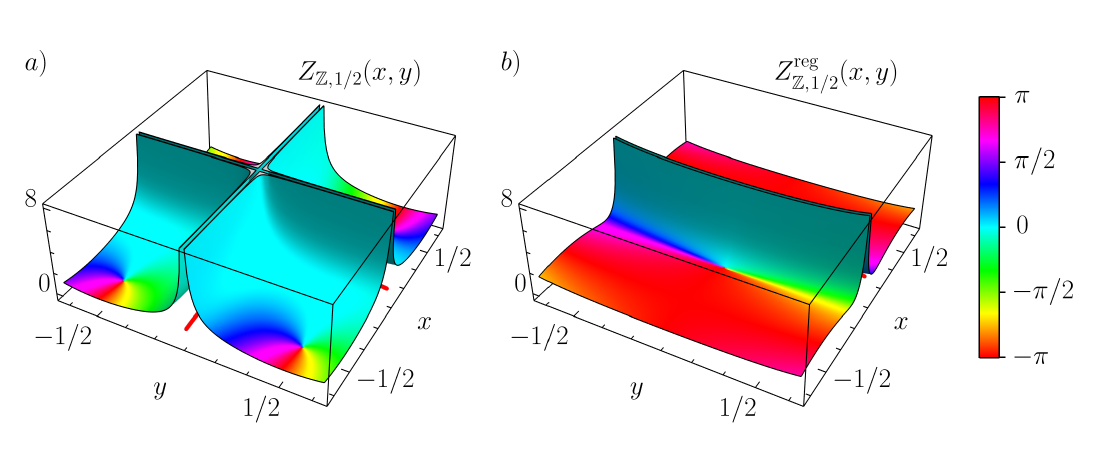}
    \caption{
    One-dimensional Epstein zeta function (a) and its regularised counterpart (b) for $\nu=1/2$ and $\Lambda=\mathds Z$,
    where the height of the surface encodes the absolute function values and the color encodes the arguments of the complex function values according to the cyclic colormap shown on the right.
    The one-dimensional Epstein zeta function exhibits discontinuities at points in the (reciprocal) lattice, in particular at $x=y=0$ (a). The regularised Epstein zeta function as a function of \(y\) is analytic in the elementary cell of the reciprocal lattice $-1/2\le y\le 1/2$ and agrees with the Epstein zeta function at $y=0$ (b).}
    \label{fig:epsteinZetaReg}
\end{figure}

The Epstein zeta function exhibits various singularities in its arguments. Understanding these singularities is of crucial importance in numerical applications that involve the Epstein zeta function as well as in physical applications. 
The singularities in $\bm x$ can readily be deduced from the complex extension in $\bm x$ of the summands of the defining lattice sum,
\[
\frac{1}{\big((\bm z-\bm x)^2\big)^{\nu/2}}\to\infty,\qquad \bm x\to \bm z,
\]
for $\mathrm{Re}(\nu)>0$. 
\new{The appearance of singularities in $\bm x\in \Lambda$ also lead to singularities in the reciprocal lattice \(\bm y\in\Lambda^*\) due to the functional equation.}
The singularity at \(\bm y=\bm 0\) is of particular interest in numerous applications of the Epstein zeta function, such as in anomalous quantum spin-wave dispersion relations, see Section \ref{sec:dispersion_relation}, 
or in superconductors with long-range interactions, see \cite{buchheit2023exact}.
For example,
investigating the latter leads to singular integrals of the form
\[\int_{E^{\ast}_{\Lambda}}\zeps{\Lambda,\nu}{\bm 0}{\bm y}f(\bm y)\,\mathrm{d}\bm y,\]
for some sufficiently regular function \(f\). In order to evaluate the integral precisely, specialised quadrature rules are required that take into account the particular form of the singularity at $\bm y = \bm 0$, see \cite{duffy1982quadrature}. \new{The singularity at $\bm y \in \Lambda^\ast$ corresponds to the Rayleigh--Wood anomaly that arises in quasi-periodic lattice sums and related Green functions in wave scattering \cite{denlinger2017fast,wood1902xlii,bruno2020evaluation}. The following treatment shows how to subtract singularities for general parameter choices while avoiding catastrophic cancellation.} 

We \new{first} show that the  Epstein zeta function can be decomposed into a power-law singularity and a regularised function that is holomorphic around $\bm y = \bm 0$. 
This is depicted for the integer-lattice and $\nu=1/2$ in 
Figure \ref{fig:epsteinZetaReg}. 
In the following, we derive a modified Crandall representation that allows for the investigation of the holomorphic part of the Epstein zeta function.
The resulting regularised Epstein zeta function plays a key role in the recently discovered singular Euler--Maclaurin expansion \cite{buchheit2022singular}, a generalisation of the classical Euler--Maclaurin summation formula to higher-dimensional lattice sums and physically relevant power-law singularities. 

We define the regularised Epstein zeta function as follows.

\begin{definition}[Regularised Epstein zeta function]
\label{def:epstienreg}
Let $\Lambda$ be a $d$-dimensional lattice, let $\bm x \in\mathds{R}^d$, let
$\bm y \in (\mathds R^d\setminus\Lambda^*)\cup\{\bm 0\}$ and let $\nu \in \mathds{C}$.
We define the regularised Epstein zeta function via 
\[\zepsr{\Lambda, \nu}{\bm x}{\bm y} = e^{2 \pi  i  \bm x \cdot \bm y} \zeps{\Lambda, \nu}{\bm x}{\bm y} - \frac{1}{V_\Lambda} \hat{s}_\nu(\bm y),\qquad \bm y\neq 0,\]
and continuously extended to $\bm y = 0$. Here $\hat s_\nu$ denotes the distributional Fourier transform of $s_\nu(\cdot)=\vert \bm \cdot\vert^{-\nu}$ which for $\nu \not\in (d+2\mathds N_0)$ is given by
\[\hat{s}_\nu(\bm y) = \frac{\pi^{\nu/2}}{\Gamma(\nu/2)}\Gamma\big((d-\nu)/2\big)  (\pi \bm y^2)^{(\nu - d)/2}.
\]
For $\nu\in d+ 2k$ and $k\in\mathds N_0$, this Fourier transform is uniquely defined up to a polynomial of order $2k$, see \cite{hormander2003analysisI}. We adopt the choice
\[
\hat s_{d+2k}(\bm y)= \frac{\pi^{k+d/2}}{\Gamma(k+d/2)}\frac{(-1)^{k+1}}{k!} ( \pi \bm y^2 )^{k} \log (\pi  \bm y^{2}).
\]
\end{definition}

The Crandall representation for the regularised Epstein zeta function
enables the analysis of its analytic properties and permits its efficient computation
without cancellation error.

\begin{theorem}[Crandall representation for the regularised Epstein zeta function]
\label{crandall-reg}
Let $\Lambda$ be a d-dimensional lattice, let $\bm x \in\mathds{R}^d$, let
$\bm y \in (\mathds R^d\setminus\Lambda^*)\cup\{\bm 0\}$ and let $\nu \in \mathds{C}$.
Then for $\lambda > 0$,
\begin{align*}
\zepsr{\Lambda,\nu}{\bm x}{\bm y}
=
&\frac{(\pi/\lambda^2)^{\nu/2}}{\Gamma(\nu/2)}
\Bigg[
\sum_{\bm{z}\in\Lambda}
\;G_{\nu}\Big(\frac{\bm{z}-\bm{x}}{\lambda}\Big)e^{-2\pi i\bm{y}\cdot(\bm{z}-\bm{x})} 
\\
&\quad
+\frac{\lambda^d}{V_{\Lambda}}
\Big[
G_{d-\nu,\lambda}^{\rm reg}(\bm y)+
\sum_{\substack{\bm{k}\in\Lambda^{\ast}\\ \bm k\neq \bm 0\,\,}}G_{d-\nu}
(\lambda(\bm{k}+\bm{y}))e^{-2\pi i \bm{x}\cdot \bm{k}}
\Big]\Bigg]
\end{align*}
where the regularised upper Crandall function is defined as
\[
G_{\nu,\lambda}^{\rm reg}(\bm y)
=
G_{\nu}(\lambda\bm y)-
\lambda^{-\nu}\pi^{-(d-\nu)/2}\Gamma((d-\nu)/2)\hat{s}_{d-\nu}(\bm y).
\]
If \(\nu\notin -2\mathds N_0\), we have
$$G_{\nu,\lambda}^{\rm reg}(\bm y)=
-g_{\nu}(\lambda \bm y) 
$$
and if \(\nu=-2k\) for \(k\in\mathds N_0\) it holds that
\[G^{\rm reg}_{\nu,\lambda}(\bm y)
=\frac{(-1)^k}{k!}(H_k-\gamma -\lambda^{2k}\log\lambda^2)(\pi \bm y^2)^k-\sum_{\substack{n=0 \\ n\neq k}}^{\infty}\frac{(-\pi \bm y^2)^n}{(n-k)n!}.
\]
Here \(\gamma\) is the Euler–-Mascheroni constant and \(H_k\) is the \(k\)-th harmonic number.
\end{theorem}

\begin{proof}
We rewrite the regularised Epstein zeta by subtracting the singularity from the \(\bm k = \bm 0\) summand in the second sum of Crandall's representation.
For ${d-\nu\notin (-2\mathds N_0)}$,
$$
G_{d-\nu,\lambda}^{\rm reg}(\bm y)
=G_{d-\nu}(\lambda \bm y) - \frac{\Gamma((d-\nu)/2)}{(\pi \lambda^2 \bm y^2)^{(d-\nu)/2}}
$$
is given in terms of the lower Crandall function by the 
fundamental relation \cite[Eq. (6.5.3)]{abramowitz}.
If $d-\nu=-2k$ for some $k\in\mathds N_0$, the representation for
$$
G_{d-\nu,\lambda}^{\rm reg}(\bm y)=
(\pi\lambda^2\bm y^2)^k\Big(\Gamma(-k,\pi\lambda^2\bm y^2)
+\frac{(-1)^{k}}{k!} (\log (\pi\lambda^2\bm y^2)-\log\lambda^2)\Big)
$$
is obtained, by
applying the series representation
\cite[Eq. (5.1.12)]{abramowitz}
$$\Gamma(-k,z)=
\frac{(-1)^{k}}{k!}(H_k-\log z)-\sum_{\substack{n=0 \\ n \neq k}}^{\infty}
\frac{(-z)^n}{(n-k)k!}
$$
see \cite[Eq. (5.1.12)]{abramowitz}
at \(z=\pi\lambda^2\bm y^2\).
\end{proof}

\new{In complete analogy, further singularities in $\bm y$, and similarly in $\bm x$, can be removed by replacing an upper Crandall function by its regularized version.}

The regularised Crandall function as introduced in Theorem \ref{crandall-reg} is entire in $\bm z$ for every fixed exponent $\nu\in\mathds C$.

\begin{lemma}
\label{g-reg-an}
Let $\nu \in \mathbb{C}$ and let $\lambda >0$. Then, $G_{\nu,\lambda}^{\rm reg}(\bm z)$ 
\new{can be holomorphically extended to} $\bm z\in\mathds C^d$.
\end{lemma}

\begin{proof}
Assume $\nu \notin -2\mathbb{N}_0$, then $G_{\nu,\lambda}^{\rm reg}(\bm z)$
is holomorphic in $\bm z\in\mathds C^d$
by 
Lemma \ref{lem:propCrandall} (1). 
Now, let $\nu =-2 k$ for $k \in \mathbb{N}_0$ and
\[f_{\lambda}(w)=\frac{(-1)^k}{k!}(H_k-\gamma -\lambda^{2k}\log\lambda^2)w^k-\sum_{\substack{n=0 \\ n\neq k}}^{\infty}\frac{(-w)^n}{(n-k)n!}.
\]
Since
a power series is holomorphic if and only if it converges absolutely and
\[G_{\nu,\lambda}^\mathrm{reg}(\bm z)=f_{\lambda}( \pi \bm z^2),\]
it suffices to show that 
$$\sum_{\substack{n=0 \\ n\neq k}}^{\infty}\frac{(-w)^n}{(n-k)n!}
$$
converges absolutely.
The absolute value of the summands for $n \neq k$ is bounded by
\[\frac{|w|^n}{|n-k|n!} \leq \frac{|w|^n}{n!}\]
and thus, the exponential function is an absolute majorant of the power series associated with
\[\frac{(-1)^k}{k!}(H_k-\gamma -\lambda^{2k}\log\lambda^2)w^k- f_{\lambda}(w).
\qedhere
\]
\end{proof}

The regularised Epstein zeta function can be holomorphically extended in $\bm y$ around zero.

\begin{theorem}[\new{Holomorphic extension} of the regularised Epstein zeta function]
Let $\Lambda$ be a $d$-dimensional lattice. Then the regularised Epstein zeta function can be holomorphically extended to \[
(\nu,\bm x, \bm y)\in \big(\mathds C\setminus (d+2\mathds N_0)\big)\times D_\Lambda\times D_{\Lambda^\ast\setminus\{\bm 0\}}.
\]
For $\nu\in (d+2 \mathds N_0)$, the regularised Epstein zeta function \new{can be holomorphically extended to} \[(\bm x, \bm y) \in D_\Lambda\times D_{\Lambda^\ast\setminus\{\bm 0\}}.
\]
Finally, for $\bm x\in \Lambda$, the regularised Epstein zeta function \new{can be holomorphically extended to}
\[
(\nu,\bm y) \in \big(\mathds C\setminus (d+2\mathds N_0)\big)\times D_{\Lambda^\ast\setminus\{\bm 0\}}.
\]
\end{theorem}

\begin{proof}
In proof of
Theorem \ref{hol}, we have shown holomorphy of every summand in Crandall's representation as well as compact uniform convergence. Thus, the domain of holomorphy of the Epstein zeta function corresponds to the intersection of the domains of holomorphy of all summands.

Removing the $\bm k=\bm 0$ summand 
$G_{d-\nu}(\lambda(\bm 0+\bm y))$
in the second sum in Crandall's representation
extends holomorphy in $\bm y$ to $D_{\Lambda^\ast\setminus \{\bm 0\}}$ and removes the potential pole at $\nu = d$ in case that $\bm y = 0$.
The regularised Epstein zeta function is then obtained by adding the regularised Crandall function
$G_{d-\nu,\lambda}^{\rm reg}(\bm y)$,
which is holomorphic in 
\[
(\nu,\bm y) \in (\mathds C\setminus (-2\mathds N_0))\times \mathds C^d,
\]
where the exclusions in $\nu$ are due to the poles of the gamma function.
In case that $\nu \in (d+2\mathds N_0)$,
the resulting summand is entire in $\bm y\in\mathds C^d$ by 
Lemma \ref{g-reg-an}. The intersection of the domains of holomorphy then yields the desired result. 
\end{proof}

The holomorphy of the
regularised Epstein zeta function
ensures exponential convergence of an interpolating polynomial with respect to the number of interpolation points \cite{berrutBarycentricLagrangeInterpolation2004}. By subsequently adding back the singular term, one obtains an efficient method for precomputing the full Epstein zeta function.

\section{Algorithm}
\label{sec:algorithm}
In this section, we present an algorithm for efficiently computing the Epstein zeta for arbitrary real parameters. While Crandall's representation forms the basis for computing the Epstein zeta function in terms of a sum over superexponentially decaying function values, many numerical challenges remain. 
In this chapter, we present \hyperref[alg]{Algorithm 1}, which evaluates the Epstein zeta function for arbitrary real parameters to full precision for dimensions $d\le 10$. The discussion of the algorithm is structured as follows.
The preprocessing of the parameters, especially the projection of the vectors to the elementary lattice cell, is discussed in 
Section \ref{sec:preproc}. 
Near poles and zeros of the gamma and Crandall functions, explicit evaluations of the Epstein zeta function as discussed in 
Section \ref{sec:special} are needed to guarantee a stable evaluation.
In 
Section \ref{sec:syms} we show, how the symmetries of the Epstein zeta function simplify its evaluation.
In 
Section \ref{sec:abwert}, our choice of the summation cutoff is justified, and techniques for reducing roundoff errors are discussed.
Finally, we outline the computation of the upper incomplete gamma function in 
Section \ref{sec:ogammaalg}.

\begin{algorithm}
\label{alg}
\caption{Computation of the Epstein zeta function.}
\begin{algorithmic}

\State \textbf{Input: }$d\in\mathds{N}_+$, $A\in \mathds{R}^{d\times d}$ regular, $\bm{x},\bm{y}\in\mathds{R}^d$, $\nu\in\mathds{R}$, \new{$r>0$}

\algskip

\Statex \textbf{1. Rescale lattice matrix and vectors}

\State   $a =V_{\Lambda}^{1/d}, \quad A \gets A/a,\quad \bm{x} \gets \bm{x}/a,\quad \bm{y} \gets a \bm{y}$

\algskip
    
\Statex \textbf{2. Project $\bm{x}$ and $\bm{y}$ to elementary lattice cells}

\State ${\bm x}' = A^{-1} {\bm x},\quad {\bm y}' = A^T {\bm y}$
    
\For{\(i = 1,\ldots,d\)}

    \State $x_i' \gets ( (x_i'+1/2) \mod{1} )-1/2$
        
    \State $y_i' \gets ( (y_i'+1/2) \mod{1} )-1/2$

\EndFor

\State ${\bm x}' \gets A {\bm x}', \quad {\bm y}' \gets A^{-T}  \bm y'$

\algskip

\Statex \textbf{3. Handle special cases}

\State \textbf{if }$\nu = d\,\land\, \bm{y}' = \bm 0\ \land$ regularised $=$ False \textbf{ then }\Return NaN \textbf{ end if}

\If{$\nu=0$}

    \State \textbf{ if } $\bm x'=\bm 0$
    \textbf{ then } 
    \Return $- \exp({-2 \pi i \bm{x} \cdot \bm{y}'})$ 
    \textbf{ end if}
    
    \State  \Return $0$

\EndIf

\State \textbf{if }$\nu\in -2\mathds N_+$\Return $0$\textbf{ end if}

\algskip

\Statex \textbf{4. Kahan summation in real and reciprocal space \new{(used for all additions within this step)}}

\State $\textnormal{sum\_real} =0, \quad \textnormal{sum\_reciprocal} =0$

\If{regularised \(=\) False}

    \For{$\bm z \in A \mathds Z^d: \vert \bm z-\bm x' \vert \le r$}
    
        \State \textnormal{sum\_real} $\pluseq G_{\nu}(\bm z-\bm x')\exp(-2\pi i\bm z\cdot \bm y')$
    
    \EndFor
    
    \For{$\bm k \in A^{-T}\mathds Z^d: \vert \bm k+\bm y' \vert\le r$}
    
        \State \textnormal{sum\_reciprocal} $\pluseq G_{d-\nu}(\bm k + \bm y')\exp(2 \pi i \bm x' \cdot (\bm k +\bm y'))$

    \EndFor
    
    \State \textnormal{sum\_real} $\asteq\exp(-2\pi i \bm y\cdot (\bm x - \bm x'))$
    
    \State \textnormal{sum\_reciprocal} $\asteq\exp(-2\pi i \bm y\cdot (\bm x - \bm x'))$

\ElsIf{regularised $=$ True}

    \For{$\bm z \in A\mathds Z^d: \vert \bm z-\bm x' \vert\le r$}
        
        \State \textnormal{sum\_real} $\pluseq 
        G_{\nu}(\bm z-\bm x')
        \exp(-2\pi i(\bm z\cdot \bm y'-\bm y\cdot \bm x'))$

    \EndFor
    
    \For{$\bm k \in (A^{-T}\mathds Z^d)\setminus \{\bm y\}: \vert \bm k+\bm y' \vert\le r$}
    
        \State \textnormal{sum\_reciprocal} $\pluseq G_{d-\nu}(\bm k + \bm y')\exp(2 \pi i \bm x \cdot (\bm k -(\bm y- \bm y')))$
        
    \EndFor
        
    \State \textnormal{sum\_reciprocal} $\pluseq G_{d-\nu,\lambda}^{\mathrm{reg}}(\bm y)$
    
\EndIf

\algskip

\Statex \textbf{5. Postprocessing}

\State $\textnormal{result} =a^{-\nu}  \pi^{\nu/2} (\textnormal{sum\_real}+\textnormal{sum\_fourier})/\Gamma(\nu/2)$

\If{regularised $=$ True $\land$ $\nu=d+2k$ for $k\in\mathds N_0$}

    $\textnormal{result}\pluseq \frac{1}{V_{\Lambda}}\frac{\pi^{k+d/2}}{\Gamma(k+d/2)}\frac{(-1)^{k+1}}{k!} ( \pi \bm y^2 )^{k} \log (a^2)$

\EndIf

\algskip

\State \Return \textnormal{result}

\end{algorithmic}
\end{algorithm}

\subsection{Preprocessing}
\label{sec:preproc}
Whenever possible, both sums in Crandall's representation should decay similarly fast, to ensure a low number of total summands to evaluate. 
To that end, we rescale the lattice matrix as in Lemma \ref{lem:syms} so that its determinant equals one.
By translational symmetry, we restrict the evaluation
to 
vectors \(\bm x\), \(\bm y\) in their respective elementary lattice cells.
The following statement allows for the projection of \(\bm x\) to its elementary lattice cell as in the second step of our algorithm.

\begin{lemma}
Let $\Lambda=A\mathds Z^d$ for $A\in\mathds R^{d\times d}$ regular, let $\bm x \in \mathds{R}^d$
and 
\[
\bm v = A\lfloor A^{-1} \bm x + \bm 1/2 \rfloor,
\]
where \(\bm 1=(1,\ldots,1)^T\in\mathds R^d\) is the one-vector and
$\lfloor \cdot\rfloor$ is the element-wise applied floor function.
Then \(\bm v\in\Lambda\) and \(\bm x-\bm v\in E_{\Lambda}\).
\end{lemma}
\begin{proof}
Since the range of $\lfloor\cdot\rfloor$ is $\mathds{Z}^d$, we find that $\bm v_x \in \Lambda$. 
The statement follows by applying
\[ -\frac 12\le w-\lfloor w+1/2\rfloor <\frac 12,\qquad w\in\mathds R,\]
in every component of
\[A^{-1}(\bm x-\bm v)=A^{-1}\bm x-\lfloor A^{-1}\bm x+\bm 1/2\rfloor.
\qedhere
\]
\end{proof}

\subsection{Special cases}
\label{sec:special}
For specific values of $\nu$, poles occur in the gamma function or in the Crandall functions. These special cases need to be treated separately in order to guarantee a stable evaluation.

In \hyperref[alg]{Algorithm 1}, when \(\bm y\in\Lambda^*\), the pole of the Epstein zeta function at \(\nu=d\) is indicated directly without evaluating the Crandall representation.
At the removable singularities \(\nu\in -2 \mathds N_0\) of the gamma function, we return the appropriate limiting value, as shown in the following lemma.

\begin{lemma}
Let $\Lambda$ be a d-dimensional lattice, $\bm x, \bm y \in \mathds{R}^d$ and \(k\in\mathds N_0\). Then
\[
\zeps{\Lambda,-2k}{\bm x}{\bm y}
=
\begin{cases}
-e^{-2\pi i\bm x\cdot \bm y} & k=0\text{ and }\bm x\in\Lambda \\
0 & \text{otherwise}
\end{cases}.
\]
\end{lemma}

\begin{proof}
Let \(\nu\in\mathds C\setminus(-2\mathds N_0)\) and \(\bm x\notin\Lambda\), then
\[G_\nu(\bm z - \bm x) e^{-2 \pi  i  \bm y \cdot \bm z}, \qquad \bm z \in \Lambda\]
and
\[G_{d-\nu}(\bm k + \bm y) e^{-2 \pi  i  \bm x \cdot ( \bm k + \bm y)}, \qquad \bm k \in \Lambda^*\]
are bounded for \(\nu\to -2k\) due to 
Lemma \ref{lem:propCrandall}, and 
\[G_{d-\nu}(\bm 0) e^{-2 \pi  i  \bm x \cdot \bm 0} = \frac{1}{d-\nu}.\]
Since \(1/\Gamma(\nu/2)\to 0\) for \(\nu\to-2k\), 
Crandall's representation implies that the Epstein zeta function 
evaluates to zero for \(\nu\in -2 \mathds N_0\).
Now let \(\bm x\in\Lambda\), so there is an additional summand of the form
\[G_\nu(\bm 0) e^{-2 \pi  i  \bm y \cdot \bm x} = -\frac{2}{\nu} e^{-2 \pi  i  \bm y \cdot \bm x}\]
in Crandall's representation.
The Epstein zeta function still evaluates to zero for \(\nu\to -2\mathds N_+\)
and 
\[-\frac{\pi^{\nu/2}}{\Gamma(\nu/2)} \frac{2}{\nu} e^{-2 \pi  i  \bm y \cdot \bm x} = -\frac{\pi^{\nu/2}}{\Gamma(\nu/2 + 1)} e^{-2 \pi  i  \bm y \cdot \bm x} \rightarrow -e^{-2 \pi  i  \bm y \cdot \bm x}
,\qquad 
\nu \to 0.
\qedhere
\]
\end{proof}

\subsection{Symmetries}
\label{sec:syms}

Our algorithm effectively exploits the symmetries of the Epstein zeta function discussed in 
Lemma \ref{lem:syms}.
The scaling symmetry allows the reduction of general lattices to the case of unit elementary cell volume, $V_\Lambda = 1$.
Further, the translation symmetry allows us to restrict all investigations to open neighborhoods of
the elementary lattice cells  $\bm x \in E_\Lambda$ and $\bm y \in E_\Lambda^*$.

Although the regularised Epstein zeta function breaks the translational invariance and inversion symmetry of the Epstein zeta function, a modified form of the scaling symmetry can be recovered.

\begin{corollary}[Scaling symmetry of the regularised Epstein zeta function]
Let \(\Lambda\) be a \(d\)-dimensional lattice,
let \(\bm x\in\mathds{R}^d\),
let
$\bm y \in (\mathds R^d\setminus\Lambda^*)\cup\{\bm 0\}$ and $\nu\in\mathds C$.
Then 
\begin{enumerate}
    \item For $s\in\mathds R\setminus\{0\}$ and $\nu\notin (d+2\mathds N_0)$, we have
$$
\zepsr{\Lambda, \nu}{\bm x}{\bm y} = |s|^\nu \zepsr{s \Lambda, \nu}{s \bm x}{\bm y / s}.
$$
    \item For $s\in\mathds R\setminus\{0\}$ and $\nu= d+2k$ for $k\in\mathds N_0$, we have
$$
\zepsr{\Lambda, \nu}{\bm x}{\bm y} = |s|^\nu \zepsr{s \Lambda, \nu}{s \bm x}{\bm y / s}-
\frac1{V_{\Lambda}}\frac{\pi^{k+d/2}}{\Gamma(k+d/2)}\frac{(-1)^{k+1}}{k!} ( \pi \bm y^2 )^{k} \log (s^2)
.
$$
\end{enumerate}
\end{corollary}
\begin{proof}
We may restrict our discussion on $\mathrm{Re}(\nu)>d$ by the uniqueness of the analytic continuation. There, the Epstein zeta function is defined via an absolutely convergent lattice sum, see 
Lemma \ref{lattice-sum-converges}. 
For $\nu\notin (d+2\mathds N_0)$ and $\bm y\neq\bm 0$, the scaling symmetry of the regularised Epstein zeta function follows from the scaling symmetry of the Epstein zeta function, 
Definition \ref{def:epstienreg}, and 
$$
\frac1{V_{s\Lambda}}\hat{s}_{\nu}(\bm y/s)
=
\frac 1{|s|^dV_{\Lambda}}|s|^{d-\nu}\hat{s}_\nu(\bm y)
$$
where the formula continuously extends to $\bm y=\bm 0$.
For $\nu=d+2k$ we obtain the correction term from
\[
\frac1{V_{s\Lambda}}\hat{s}_{d+2k}(\bm y/s)
=
\frac1{|s|^{d}V_{\Lambda}}|s|^{-2k}\Big(\hat{s}_{\nu}(\bm y)
-
 \frac{\pi^{k+d/2}}{\Gamma(k+d/2)}\frac{(-1)^{k+1}}{k!} ( \pi \bm y^2 )^{k} \log (s^2)
 \Big).
\qedhere
\]
\end{proof}

\subsection{Truncation and compensated summation}
\label{sec:abwert}

In this subsection, we present a rigorous error bound when truncating the sums in Crandall's representation. The proof is provided in Appendix \ref{sec:appendix-trunc}.

\begin{theorem}[Remainder estimate]
  \label{theorem:trunc}
Let $\Lambda=A\mathds Z^{d\times d}$, for $A\in\mathds R^{d\times d}$ regular and $\det(A)=1$. Let $\bm x, \bm y \in \mathds{R}^d$ and let
$\nu\in I$ with $I\subset\mathds R$ compact so that $\nu\neq d$ if $\bm y\in\Lambda^*$.
Let $\mathcal{R}_{\Lambda}(r)$ denote the remainder obtained by truncation at $r>0$ in Crandall's representation
\begin{align*}
\zeps{\Lambda,\nu}{\bm{x}}{\bm{y}}
=&\frac{\pi ^{\nu/2}}{\Gamma(\nu/2)}\Bigg[\sum_{\substack{\bm z\in\Lambda\\ |\bm z - \bm x|\le r}}
G_{\nu}(\bm{z} -\bm x)e^{-2\pi i\bm{y}\cdot\bm{z}}
\\
&
\quad
+\frac{1}{V_{\Lambda}}
\sum_{\substack{\bm k\in\Lambda^{\ast}\\ |\bm k + \bm y|\le r}}G_{d-\nu}
(\bm{k} + \bm y) e^{2\pi i\bm{x}\cdot (\bm{k} + \bm y)}
\Bigg]+\mathcal R_\Lambda(r).
\end{align*}
Then, for any $\varepsilon>0$ and $r>(1+2\varepsilon)\sqrt{d}\,\kappa(A)$ with the condition number $\kappa(A)=\|A\|\|A^{-1}\|$, where $\|\cdot\|$ denotes the spectral norm,
it holds that
$$
|\mathcal R_\Lambda(r)|<\kappa(A)^{d+1} \sup_{\nu \in I}
c_\nu
\Big(
R_{\nu}\big(r/\kappa(A)\big)
+
R_{d-\nu}\big(r/\kappa(A)\big)
\bigg)
$$
with the prefactor \[c_\nu = \frac{(3/2)^d\pi^{(\nu+d)/2}}{\Gamma(d/2+1)|\Gamma(\nu/2)|},\]
and with the function
$$
R_{\nu}(r)
=\frac {r^{d+1}}\varepsilon
\Big(
\frac{G_{d+1}(r-\varepsilon) - 
G_{\nu}(r-\varepsilon)}{d+1-\nu}\Big),
$$
where the limit $\nu\to d+1$ is well-defined.
In particular, 
\[
\mathcal R_\Lambda(\kappa(A)r)
\le \kappa(A)^{d+1}
\mathcal R_{\mathds Z^d}(r).\]
\end{theorem}

For practical purposes, a sufficient upper bound to the error in 
Theorem \ref{theorem:trunc} is obtained by choosing
$\varepsilon=1/20$.
We present sufficient truncation values $r=r_0$ for achieving machine precision for the square lattice $\Lambda=\mathds Z^d$ as a function of dimension $d$ in 
Table \ref{tab:square-lattice-trunc-values}.
The corresponding value for general lattices $\Lambda=A\mathds Z^d$ 
where $\det(A)=1$ 
is then obtained as
\[
r= \kappa(A) r_0,
\]
where the error lies below machine precision
if $\kappa(A)^{d+1}\le 10^2$
for the values $r_0$ in  
Table \ref{tab:square-lattice-trunc-values}.

In infinite precision arithmetic, a truncation value based on 
Theorem \ref{theorem:trunc} guarantees that the desired precision is reached when computing the Epstein zeta function.  However, in floating point arithmetic, roundoff errors accumulate in the evaluation of the sums for large dimensions $d$, which can exceed the rigorous truncation error bound by several orders of magnitude. A practical solution that significantly reduces the roundoff error is provided by the compensated summation algorithm developed by Kahan \cite{kahan}. Here, a separate compensator variable is introduced, which stores an estimation of the accumulated error. Using compensated summation, we show in 
Section \ref{sec:experiments} that our resulting algorithm reaches machine precision across all test cases, even for large dimensions $d$. \new{An alternative approach involves dividing the summation ball into shells and performing the summation in reverse order, starting from the outermost shell with the smallest absolute values and progressing toward the inner shells. In our implementation, however, we prefer Kahan compensated summation, since it is straightforward to apply in the natural traversal order of the lattice points and its additional arithmetic cost is negligible compared with the evaluation of the incomplete gamma functions.}

\begin{table}
\centering
\begin{tabular}{c | *{16}{c}}
\toprule
$d$ & 1 & 2 & 3 & 4 & 5 & 6 & 7 & 8 & 9 & 10  \\
\midrule
$r_0$  &3.8 &3.9 &4. &4.1 &4.2 &4.2 &4.3 &4.4 &4.4 &4.5 \\
\bottomrule
\end{tabular}
\caption{Truncation values \(r_0\) to achieve a truncation error $\mathcal R_{\mathds Z^{d}}<10^{-18}$ in Crandall's representation for and $-10\le \nu\le 10$ 
obtained by 
Theorem \ref{theorem:trunc} for the choice of $\varepsilon = 1/20$.}
\label{tab:square-lattice-trunc-values}
\end{table}

\subsection{Evaluation of the incomplete gamma functions}
\label{sec:ogammaalg}

A stable, fast, and accurate evaluation of the incomplete gamma function is the foundation of any algorithm relying on Crandall's representation of the Epstein zeta function.
A wide range of algorithms \new{focus} on evaluating $\Gamma(\nu, x)$ at the restricted parameter range $\nu > 0$, see for example \cite{abergel2020algorithm,gil}.
Implementations in arbitrary precision libraries such as Arb and FLINT \cite{arb} guarantee correct results but are not usable in high-performance applications due to long evaluation times.
Fast and stable implementations based on the work of \cite{gautschi} only offer single precision.
Therefore, no existing implementation is sufficient for a reliable, high-performance implementation of the Epstein zeta function.
We address this problem and provide a fast and precise implementation based on the work of \cite{gautschi,gil} that evaluates the incomplete Gamma function to full precision over the whole parameter range. In the following, we discuss the underlying methods. \new{We note that a robust implementation of the incomplete gamma functions is available in Julia through \cite{SpecialFunctions_jl}, following \cite[§8]{NIST:DLMF}. However, this library is not suitable in our setting, as embedding a scripting-language runtime into a high-performance C library would lead to significant overhead and compromise performance.}

\begin{figure}
    \centering        
    \includegraphics[width=1\linewidth]{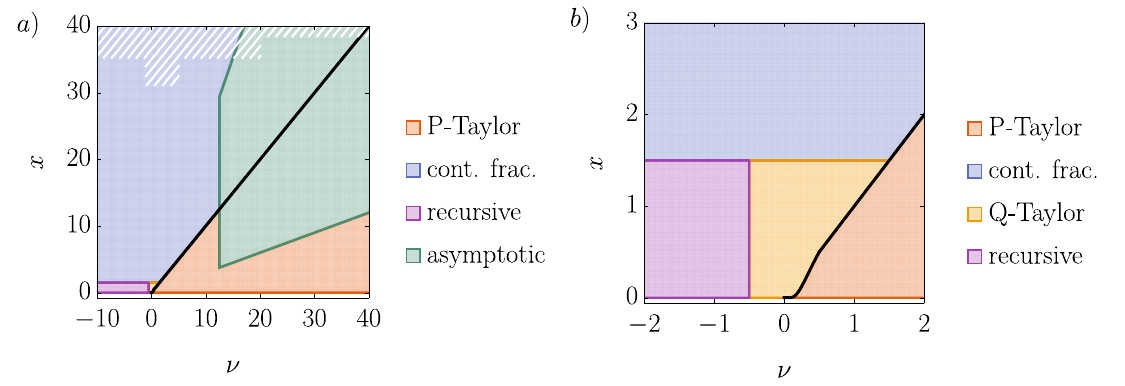}
    \caption{Regions of different evaluation methods of the incomplete Gamma functions (a).
    Panel (b) offers a magnified view of the region close to the origin.
\new{The dashed white region in (a) indicates where we use an asymptotic expansion to evaluate the Crandall function, which is otherwise computed via the incomplete gamma function.}    }
\label{fig:gammareg}
\end{figure}

The algorithm is based on evaluating the regularisation of the incomplete Gamma functions in the first parameter.

\begin{definition}
Let $x, \nu > 0$. Then, we define the incomplete gamma function ratios
\[P(\nu, x) = \frac{\gamma(\nu, x)}{\Gamma(\nu)}, \qquad Q(\nu, x) =  \frac{\Gamma(\nu, x)}{\Gamma(\nu)}.\]
\end{definition}

Depending on the values of $x$ and $\nu$, different methods for computing these functions need to be applied. The regions corresponding to the different evaluation methods as functions of 
$x$ and $\nu$ are illustrated in 
Figure \ref{fig:gammareg}.

By construction,
the range of both $P$ and $Q$ equals $[0, 1]$ and we have $P + Q = 1$.
We aim to compute $P$ whenever $Q>P$, and $Q$ if $Q<P$. 
The black line in 
Figure \ref{fig:gammareg} indicates an approximation of this separation, where we
compute $P$ whenever
\[
x <
\begin{cases}
2^{1-1/\nu}&0 < \nu < \frac{1}{2}\\
\nu &\nu \geq \frac{1}{2}
\end{cases}
\]
and compute $Q$ otherwise.

The computation of $P$ (blue region) is performed via the power series
\[P(\nu, x) = \frac{x^\nu e^{-x}}{\Gamma(\nu + 1)} \sum_{n = 0}^\infty \frac{x^n}{(\nu + 1)_n},\]
which can be derived through integration by parts. 
Here, 
$$(\nu + 1)_n = \Gamma(\nu + n + 1) / \Gamma(\nu + 1)$$ 
denotes the rising Pochhammer symbol.
The above Taylor series converges for $\nu > -1$ and arbitrary $x$. 

From this power series, we can derive a power series for $Q$ (green region). 
We write $Q = 1-P = u+v$, where
\[u = 1 - \frac{1}{\Gamma(\nu+1)} + \frac{1-x^\nu}{\Gamma(\nu+1)}\]
and
\[v = \frac{x^\nu}{\Gamma(\nu+1)} (1 - \Gamma(\nu+1) x^{-\nu} P(\nu, x)).\]
The term $v$ can be evaluated using the power series above. 
For the term $u$, we evaluate the parts
\[1 - \frac{1}{\Gamma(\nu+1)}\]
and
\[\frac{1-x^\nu}{\Gamma(\nu+1)}\]
separately using their Taylor series expansions.
Details can be found in \cite{gautschi}.
This method works for arbitrary $x$ and $\nu > -1$, and the relative error grows with $x$.
We apply this method of computing $Q$ for $\nu > -1/2$ and $x < x_0 = 3/2$.

Another method to compute $Q$ is given by the continued fraction expansion (orange region),
\[x^{-\nu} e^x \Gamma(\nu, x) = \frac{1}{x+} \frac{1-\nu}{1+} \frac{1}{x+} \frac{2-\nu}{1+} \frac{2}{x+} \frac{3-\nu}{1+} \frac{3}{x+} \dots\]
attributed to Legendre. A proof can be found in \cite{wall}. Details for evaluating the continued fraction can be found in \cite{gautschi}. This continued fraction corresponds to an asymptotic expansion at 
$x \to \infty$ and thus, it converges fast for large $x$. However, the convergence deteriorates for $x \approx \nu$.

For large $x$ and large $\nu$, we use the uniform asymptotic expansion (grey region) derived in \cite{temme1}. Details on the computation can be found in \cite{gil} and \cite{temme3}.

For negative $\nu$ and $x < x_0$, we use a recurrence relation (red region). Here, we make use of the regularisation
\[G(\nu, x) = e^x x^{-\nu} \Gamma(\nu, x).\]
Integration by parts gives the recurrence relation
\[G(-n+\varepsilon, x) = \frac{1}{n-\varepsilon} (1 - x G(-n+1+\varepsilon, x)),\]
where we assume $-1/2 \leq \varepsilon \leq 1/2$ and $n \in \mathds{N}_+$. $G(\varepsilon, x)$ can be evaluated using the methods described above. A stability discussion of this recursion pattern can be found in \cite{gautschi}; for $x_0=3/2$, the relative error of $G(\varepsilon, x)$ is amplified by at most $5.7$.

\new{
The error of our algorithm in comparison with high-precision reference values is shown in Figure~\ref{fig:gammatm}.
There, we compute the incomplete Gamma function for
$-12.5\le\nu\le12.5$ and $0<x\le 20$ in increments of $\Delta\nu=\Delta x=2^{-4}$ (a), and for $0<x<2$ in increments of $\Delta x=2^{-7}$ (b).
The error is flat and bounded everywhere by $3\cdot 10^{-15}$ with a median error of approximately $2.8\cdot 10^{-17}$.
The evaluation time
on an Intel Core i7-1260P processor for all of these values is smaller than $8\cdot 10^{-7}$ seconds, while the median evaluation time is approximately $1.4\cdot 10^{-7}$ seconds.
In comparison, the median evaluation time of the incomplete gamma function to double precision in the arbitrary-precision C library Arb/Flint \cite{arb} over the same values on the same machine is $9\cdot 10^{-6}$ seconds, with a maximum evaluation time of $2\cdot 10^{-4}$ seconds. Using arbitrary-precision libraries would therefore significantly increase the runtime of our algorithm.
The typical evaluation time of the incomplete gamma function to double precision in Mathematica and in the \texttt{mpmath} Python package is around $10^{-5}$ seconds.
}

To compute the lower incomplete gamma function, we can use the previous considerations to get a modified algorithm. Instead of using the series expansion for the functions $P$ and $Q$, we directly evaluate $\gamma^*$ with a similar series. For negative $\nu$, specialised evaluation is required at $\nu \in -\mathds{N}_+$ and at $x=0$. It turns out to be advantageous to avoid the recurrence relation around $\nu = -1/2$ for small $x$. Here, we also use the series expansion for $\gamma^*$.

\new{We note recent developments on $\mathcal O(1)$ algorithms for the evaluation of the upper incomplete gamma function in \cite{greengard2019algorithm}. However, these methods are not yet applicable in our setting, as they are currently restricted to positive arguments.}

\new{
The above considerations are used to evaluate the Crandall functions. For vector arguments  with sufficiently large norm, we employ the asymptotic expansion of the upper Crandall function
$$
G_{\nu}(\bm z)\approx
\frac{e^{-\pi \bm z^2}}{\pi \bm z^2}
\Big(
1+\frac{\nu-2}{2\pi\bm z^2}
\Big)
,\qquad |\bm z|\gg 0,
$$
which follows directly from \cite[6.5.32]{abramowitz}. Specifically, we use this expansion in the regions $\nu\in (-1,5)$ and $|\bm z|\ge 3.15$, or $\nu\in (-20,20)$ and $|\bm z|>3.35$, or $\nu\in(-150,60)$ and $|\bm z|>3.5$. This region was determined numerically so that both the absolute and relative errors between the expansion and the exact value remain below $10^{-18}$ and is displayed by the dashed white lines in Figure~\ref{fig:gammareg} (a).
}

\begin{figure}
    \centering        
    \includegraphics[width=1\linewidth]{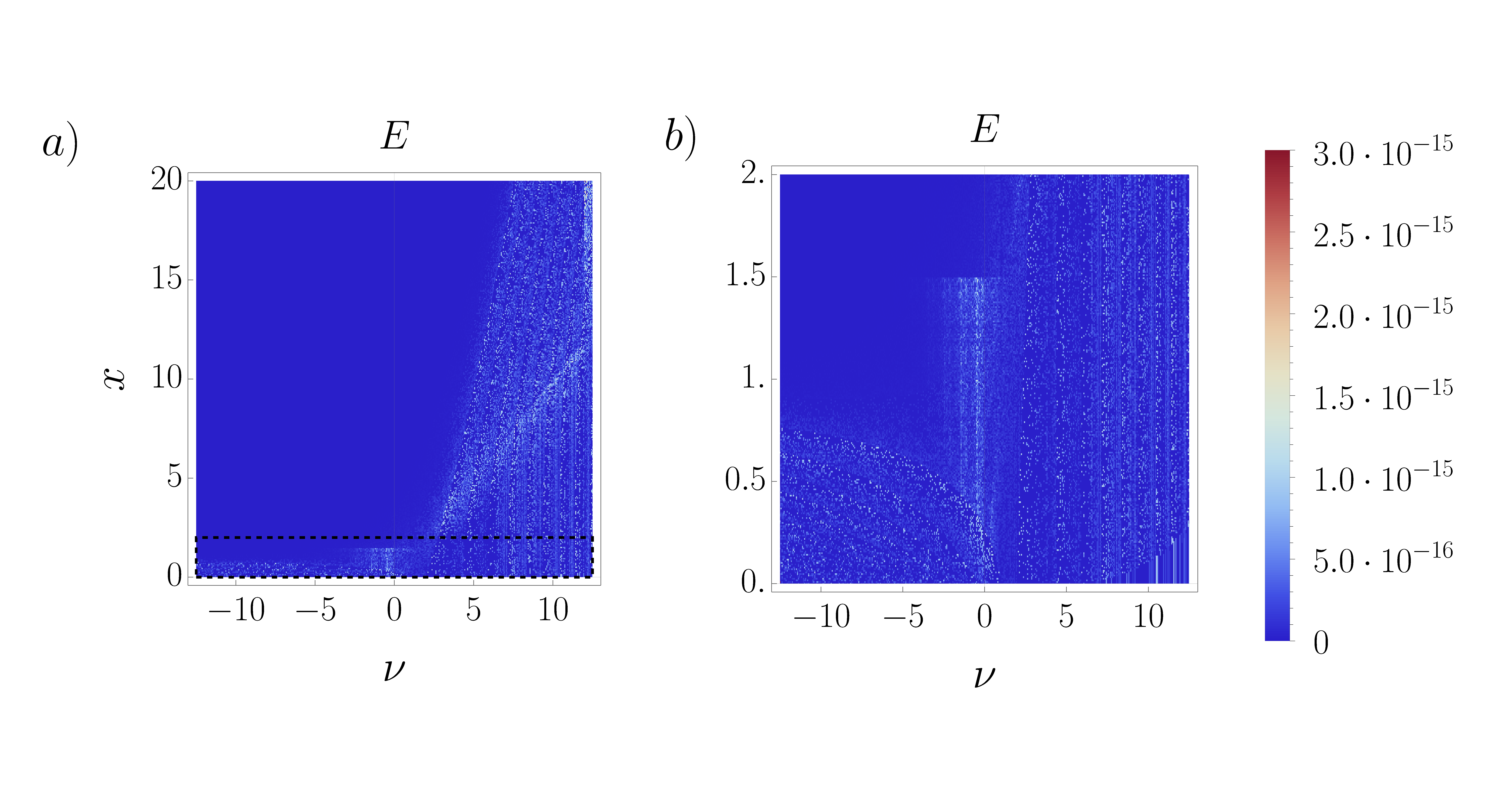}
    \caption{Minimum of absolute and relative error of our implementation of the upper incomplete gamma function compared to the arbitrary precision result from \cite{Johansson2017arb}. The dashed black box in (a) outlines the magnified region shown in (b).}
\label{fig:gammatm}
\end{figure}

\section{Numerical experiments}
\label{sec:experiments}

\begin{table}
\renewcommand{\arraystretch}{1.3}
\centering
\resizebox{\textwidth}{!}{%
\begin{tabular}{c|c|c|c|c|l}
    \toprule
    Sum & \(d\) & \(A\) &\(\bm x\) & \(\bm y\)  & value \\
    \midrule
    \(S_1\) & \(1\) & \(\mathds 1\) & \(-\bm e^{(1)}/2\) & \(\bm 0\)  
                & \(2 \zeta(\nu,1/2)\)\\
    \(S_2^{(1)}\) & \(2\) & \(\operatorname{diag}(1,2)\) & \(-\bm e^{(1)}-2\bm e^{(2)}\) & \(\bm 0\)  
                & \(2 (1 - 2^{-\nu/2} + 2^{1 - \nu}) \)\\[-0.5ex]
                & & & & & \(~\times \zeta(\nu/2)\beta(\nu/2) \) \\[-1ex]
    \(S_2^{(2)}\) & \(2\) & \(\left(\begin{matrix}1 & 1/2 \\ 0 & \sqrt{3}/2\end{matrix}\right)\) & \(\bm 0\) & \(\bm 0\)  
                & \(3^{1 - \nu/2}2 \zeta(\nu/2)\) \\[-2.5ex]
                & & & & & \(~\times(\zeta(\nu/2, 1/3) - \zeta(\nu/2, 2/3))\)\\
    \(S_3^{(1)}\) & \(3\) & \(\operatorname{diag}(1,1,2)\) & \(-\bm e^{(3)}/2\) & \(\bm e^{(1)}/2\)  
                & \(4^{\nu/2}\beta(\nu-1)\)\\
    \(S_3^{(2)}\) & \(3\) & \(\operatorname{diag}(6,6,6)\) & \(-\bm 1\) & \(\bm 1/12\)  
                & \(3^{-\nu/2} \beta(\nu - 1)\)\\
    \(S_3^{(3)}\) & \(3\) & \(\operatorname{diag}(2\sqrt{2},4,2)\) & \(-\bm e^{(2)}-\bm e^{(3)}\) & \(\bm e^{(1)}/(4\sqrt{2})\)  
                & \(2^{1 - \nu/2} \beta(\nu - 1)\)\\
    \(S_4\) & \(4\) & \(\mathds 1\) &\(\bm e^{(1)}/2\) & \(\bm 0\) 
                & \(2^\nu(\beta(\nu/2)\beta(\nu/2-1)\)\\
    & & & & & \(~+\lambda(\nu/2)\lambda(\nu/2-1))\)\\
    \(S_6\) & \(6\) & \(\mathds 1\) &\(\bm 0\) & \((\bm e^{(1)}+\bm e^{(2)})/2\) 
                & \(4 \beta(\nu/2 - 2)\eta(\nu/2)\)\\
    \(S_8\) & \(8\) & \(\mathds 1\) &\(\bm 0\) & \(\bm 1/2\)
                & \(-16 \eta(\nu/2 - 3)\zeta(\nu/2)\)\\
     \bottomrule
\end{tabular}
}
\caption{
Analytic formulas for special cases of the Epstein zeta function for \(\Lambda =A\mathds{Z}^{d\times d}\) in \(d=1,2,4,6,8\) dimensions.
Here, \(\mathds 1\in\mathds{R}^{d\times d}\) is the unit matrix, \(\operatorname{diag}(a_1,\ldots,a_d)\in\mathds{R}^{d\times d}\) is the diagonal matrix with diagonal entries \(a_1,\ldots,a_d\), the unit vector \(\bm e^{(i)}\in\mathds{R}^d\) has the components \(e^{(i)}_j=\delta_{ij}\) for \(1\le i,j\le d\) and the one-vector is defined as \(\bm 1 = (1,\ldots,1)^T\in\mathds{R}^{d}\).}
    \label{tab:defanalyticsums}
\end{table}

\begin{figure} 
    \centering
\includegraphics[width=1\linewidth]{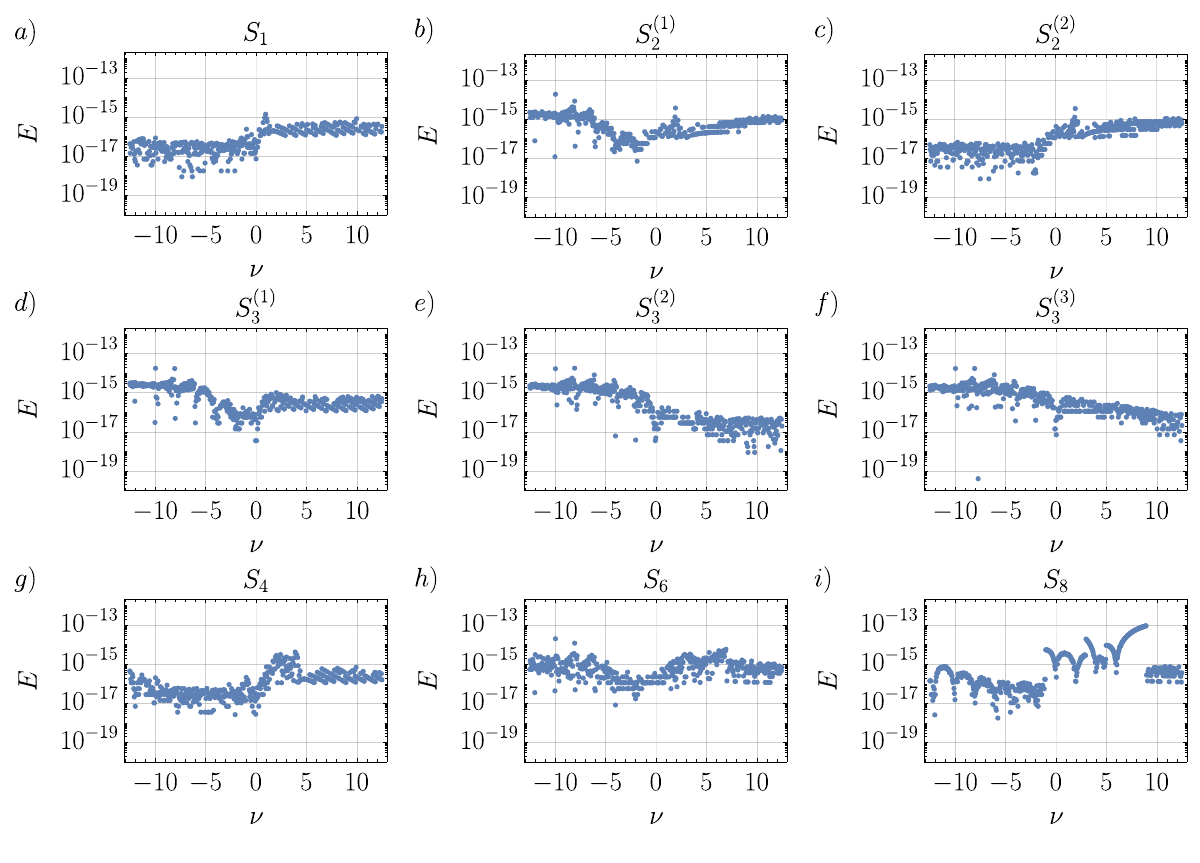}
    \caption{Minimum of absolute and relative error of the Epstein zeta function for the values as in 
    Table \ref{tab:defanalyticsums}.
        The evaluation times of the regularised function show similar behavior.
    }
\label{fig:epsteinerror}
\end{figure}

\begin{table}
\begin{tabular}{cc|cc|cc}
\toprule
 \multicolumn{2}{c|}{$d=1$} & \multicolumn{2}{c|}{$d=2$} & \multicolumn{2}{c}{$d=3$} \\
\cline{1-6}
$E_{\rm med}$ & $E_{\rm max}$ & $E_{\rm med}$ & $E_{\rm max}$ & $E_{\rm med}$ & $E_{\rm max}$\\
\toprule

$1.11\cdot 10^{-16}$ & $1.25\cdot 10^{-15}$ & $1.11\cdot 10^{-16}$ & $1.67\cdot 10^{-15}$ & $8.78\cdot 10^{-17}$ & $2.65\cdot 10^{-15}$ \\

\bottomrule
\end{tabular}
 \caption{
 \new{
 Minimum of absolute and relative error $E$ of
 $Z_{\mathds Z^d,\nu}(\bm x, \bm y)$ in comparison with an arbitrary precision implementation of our method, where for each fixed $d\in\{1,2,3\}$,
the median and maximum of the errors over a grid of values $-1/2\le \nu\le 3/2$ and
 $\bm y^T=(y,\ldots,y)$ and $\bm x^T=(x,\ldots,x)$ for $0\le x,y\le 1/2$ in increments of $\Delta\nu=1/4$ and $\Delta y=\Delta x=1/10$ is shown, while omitting the singularity at $(\nu,\bm y)=(d,\bm 0)$.
 }
    }
\label{tab:stability}
\end{table}

\begin{figure} 
    \centering
        \includegraphics[width=1\linewidth]{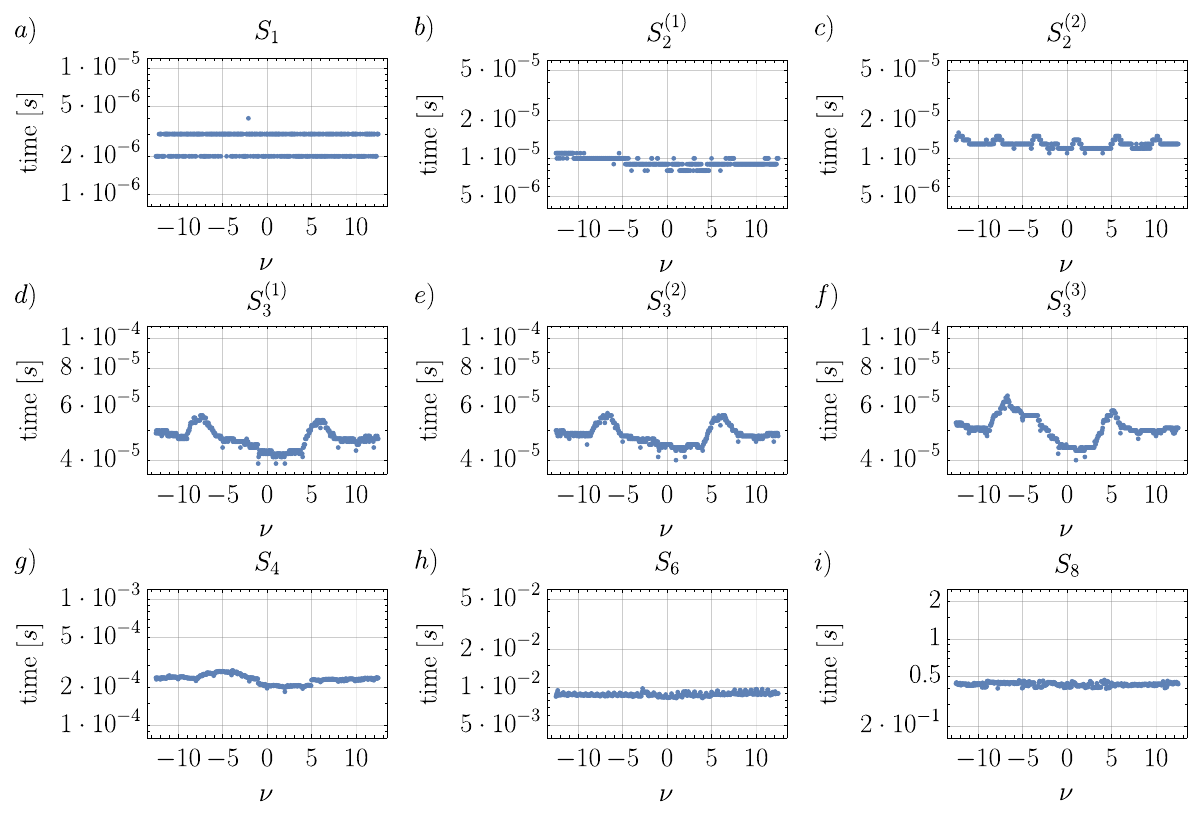}
    \caption{Evaluation times of the Epstein zeta function 
    of the Epstein
zeta function for the values as in 
Table \ref{tab:defanalyticsums}.
    The evaluation times of the regularised function are approximately equal, see 
    Table \ref{fig:perfacc}.}
    \label{fig:zetatiming}
\end{figure}

\begin{table}
\resizebox{\linewidth}{!}{%
\renewcommand{\arraystretch}{1.3}
\centering
\begin{tabular}{l|c|c|c|c|c|c|c|c|c}
    \toprule
& $S_1$ & $S_2^{(1)}$ & $S_2^{(2)}$ & $S_3^{(1)}$ & $S_3^{(2)}$ & $S_3^{(3)}$ & $S_4$ & $S_6$ & $S_8$ \\
    \midrule

\shortstack[l]{Summand \\ functions}
 & \shortstack[c]{$14$ \\ $90.5\%$} 
 & \shortstack[c]{$110$ \\ $93.3\%$} 
 & \shortstack[c]{$242$ \\ $92.4\%$} 
 & \shortstack[c]{$680$ \\ $91.7\%$} 
 & \shortstack[c]{$686$ \\ $91.8\%$} 
 & \shortstack[c]{$770$ \\ $91.2\%$} 
 & \shortstack[c]{$4.80 \cdot 10^{3}$ \\ $88.8\%$} 
 & \shortstack[c]{$2.35 \cdot 10^{5}$ \\ $84.8\%$} 
 & \shortstack[c]{$1.15 \cdot 10^{7}$ \\ $87.8\%$} 
\\

\midrule

\shortstack[l]{Incomplete \\ gamma}
 & \shortstack[c]{$12$ \\ $78.3\%$} 
 & \shortstack[c]{$68$ \\ $74.9\%$} 
 & \shortstack[c]{$69$ \\ $55.9\%$} 
 & \shortstack[c]{$294$ \\ $62.5\%$} 
 & \shortstack[c]{$302$ \\ $62.3\%$} 
 & \shortstack[c]{$297$ \\ $61.6\%$} 
 & \shortstack[c]{$1.13 \cdot 10^{3}$ \\ $45.9\%$} 
 & \shortstack[c]{$1.36 \cdot 10^{4}$ \\ $14.9\%$} 
 & \shortstack[c]{$1.12 \cdot 10^{5}$ \\ $2.7\%$} 
\\

\bottomrule
\end{tabular}
}
\caption{%
\new{
Number of calls and fraction of CPU time spent in the different functions of \hyperref[alg]{Algorithm 1} for the benchmark functions defined in Table~\ref{tab:defanalyticsums}.
Each entry gives the total number of calls and the percentage of the total runtime for the sums in Table~\ref{tab:defanalyticsums}, evaluated for a range of values $-12.5\le \nu\le 12.5$ similar to Figure~\ref{fig:epsteinerror} and for $d\in\{1,2,3,4,6,8\}$, rounded to three significant digits. For each input, the computation was repeated an appropriate number of times depending on the dimension.
The summand functions refer to the evaluation of the summands in real and reciprocal space on the right-hand side of $\texttt{sum\_rel}\pluseq$ and $\texttt{sum\_reciprocal}\pluseq$ in step 4 of \hyperref[alg]{Algorithm 1}, including the Crandall function, the exponential factor, and the linear algebra required to compute the lattice vectors $\bm z\in A\mathds Z^d$ and $\bm k\in A^{-T}\mathds Z^d$.
As the Crandall function is evaluated using an asymptotic formula for large second arguments, it is called more often than the incomplete gamma function.
}
}
    \label{tab:percentages}
\end{table}

This section provides a detailed analysis of the error and runtime of our algorithm for an extensive set of parameters.
\new{
To the best of our knowledge, no other implementation of the Epstein zeta function for non-trivial arguments in $d>1$ dimensions exists.}
\new{To reliably }benchmark the error of our algorithm, known formulas for special cases of the lattice sums
\[
S_d=\zeps{\Lambda,\nu}{\bm x}{\bm y}
\]
from \cite{crandall2012unified,zucker2017exact,burrowsLatticeSumHexagonal2023} 
in $d=1,2,3,4,6,8$ dimensions are collected. 
The particular choices for $\bm x$, $\bm y$, and $A$, as well as the resulting analytic formulas for elementary functions, are provided in 
Table \ref{tab:defanalyticsums}. 
Here $\eta$ denotes the Dirichlet eta function
\[\eta(\nu)=(1-2^{1-\nu})\zeta(\nu),\]
$\lambda$ the Dirichlet lambda function, 
\[\lambda(\nu)=(1-2^{-\nu})\zeta(\nu)\]
and $\beta$ the Dirichlet beta function,  
\[\beta(\nu)
=4^{-\nu}(\zeta(\nu,1/4)-\zeta(\nu,3/4)).\]
The function \(\zeta(\cdot)\) with one complex argument is the Riemann zeta function and
for any $x>0$, the function \(\zeta(\cdot,x)\) denotes the Hurwitz zeta function, given by the meromorphic continuation of
\[\zeta(\nu,x)=\sum_{n=0}^{\infty}\frac{1}{(n+x)^{\nu}}
,\qquad
\Re\nu>1\]
to \(\nu\in\mathds{C}\),
see \cite[Section §25.11(i)]{NIST:DLMF}.
Note that the regularised Epstein zeta function follows from Def.~\ref{def:epstienreg} and agrees with the Epstein zeta function in the case $\bm y = 0$ if $\nu\neq d$. 
We evaluate the \new{analytic formulas} for $\nu=-12.5+\varepsilon,\dots,12.5+\varepsilon$ for \(\varepsilon = 2^{-15}\) in increments of 0.05 in Mathematica with 200 digits and measure the minimum of relative and absolute error \[E=\mathrm{min}(E_{\mathrm{abs}}, E_\mathrm{rel})\] of our implementation both for the normal and the regularised Epstein zeta function.
The offset \(\varepsilon\) is chosen to avoid instant evaluation in the special cases of our algorithm, which would distort the timing benchmarks.
The runtime and precision of our algorithm for both the standard and regularised Epstein zeta function are summarised in 
Figure \ref{fig:perfacc}.

We display the error as a function of \(\nu\) in Figure \ref{fig:epsteinerror}.
We obtain full precision over the complete $\nu$-range, with the maximum error rising mildly from \new{$1.6\times 10^{-15}$ for $d=1$ to $9.1\times 10^{-14}$} for large dimension $d=8$. The small increase in the error of the regularised Epstein zeta function around $\nu=d+2\mathds N_0$ is unavoidable and results from the singularities of the Fourier transform of $\vert \bm \cdot \vert^{-\nu}$.

\begin{table}
\centering
\begin{tabular}{cccccc}
    \toprule
    sum & type & \(E_{\mathrm{max}}\) & \(t_{\mathrm{min}}-t_{\mathrm{max}}\,[s]\) & \(\new{t_{\mathrm{med}}\,[s]}\)\\
\midrule
\midrule
\multirow{2}{*}{$S_1$} & non-reg & $1.4\cdot 10^{-15}$ & $2.0\cdot 10^{-6}$$-$$4.0\cdot 10^{-6}$ & $3.0\cdot 10^{-6}$ \\ 
			         & reg & $1.6\cdot 10^{-15}$ & $2.0\cdot 10^{-6}$$-$$4.0\cdot 10^{-6}$ & $3.0\cdot 10^{-6}$ \\

\midrule
\multirow{2}{*}{$S_2^{(1)}$} & non-reg & $1.8\cdot 10^{-14}$ & $8.0\cdot 10^{-6}$$-$$1.1\cdot 10^{-5}$ & $9.0\cdot 10^{-6}$ \\ 
			         & reg & $1.8\cdot 10^{-14}$ & $8.0\cdot 10^{-6}$$-$$1.2\cdot 10^{-5}$ & $1.0\cdot 10^{-5}$ \\

\midrule
\multirow{2}{*}{$S_2^{(2)}$} & non-reg & $3.4\cdot 10^{-15}$ & $1.1\cdot 10^{-5}$$-$$1.6\cdot 10^{-5}$ & $1.3\cdot 10^{-5}$ \\ 
			         & reg & $3.4\cdot 10^{-15}$ & $1.1\cdot 10^{-5}$$-$$1.6\cdot 10^{-5}$ & $1.3\cdot 10^{-5}$ \\

\midrule
\multirow{2}{*}{$S_3^{(1)}$} & non-reg & $1.7\cdot 10^{-14}$ & $3.9\cdot 10^{-5}$$-$$5.6\cdot 10^{-5}$ & $4.7\cdot 10^{-5}$ \\ 
			         & reg & $1.8\cdot 10^{-14}$ & $3.9\cdot 10^{-5}$$-$$5.6\cdot 10^{-5}$ & $4.7\cdot 10^{-5}$ \\

\midrule
\multirow{2}{*}{$S_3^{(2)}$} & non-reg & $1.8\cdot 10^{-14}$ & $4.0\cdot 10^{-5}$$-$$5.7\cdot 10^{-5}$ & $4.8\cdot 10^{-5}$ \\ 
			         & reg & $1.8\cdot 10^{-14}$ & $4.1\cdot 10^{-5}$$-$$5.7\cdot 10^{-5}$ & $4.8\cdot 10^{-5}$ \\

\midrule
\multirow{2}{*}{$S_3^{(3)}$} & non-reg & $1.7\cdot 10^{-14}$ & $4.0\cdot 10^{-5}$$-$$6.5\cdot 10^{-5}$ & $5.1\cdot 10^{-5}$ \\ 
			         & reg & $1.7\cdot 10^{-14}$ & $4.0\cdot 10^{-5}$$-$$6.4\cdot 10^{-5}$ & $5.1\cdot 10^{-5}$ \\

\midrule
\multirow{2}{*}{$S_4$} & non-reg & $4.1\cdot 10^{-15}$ & $1.9\cdot 10^{-4}$$-$$2.7\cdot 10^{-4}$ & $2.3\cdot 10^{-4}$ \\ 
			         & reg & $4.1\cdot 10^{-15}$ & $1.8\cdot 10^{-4}$$-$$2.7\cdot 10^{-4}$ & $2.4\cdot 10^{-4}$ \\

\midrule
\multirow{2}{*}{$S_6$} & non-reg & $2.0\cdot 10^{-14}$ & $8.2\cdot 10^{-3}$$-$$9.8\cdot 10^{-3}$ & $8.8\cdot 10^{-3}$ \\ 
			         & reg & $2.1\cdot 10^{-14}$ & $8.1\cdot 10^{-3}$$-$$1.1\cdot 10^{-2}$ & $8.8\cdot 10^{-3}$ \\

\midrule
\multirow{2}{*}{$S_8$} & non-reg & $9.1\cdot 10^{-14}$ & $4.0\cdot 10^{-1}$$-$$4.7\cdot 10^{-1}$ & $4.3\cdot 10^{-1}$ \\ 
			         & reg & $2.2\cdot 10^{-14}$ & $4.0\cdot 10^{-1}$$-$$4.7\cdot 10^{-1}$ & $4.3\cdot 10^{-1}$ \\
     \bottomrule
\end{tabular}
\caption{Performance and accuracy of the algorithm, compared to analytic representations as in 
Table \ref{tab:defanalyticsums}\new{, including the minimum, maximum and median evaluation times}.
The lower index $d$ of the sum $S_d^{(j)}$ indicates the lattice dimension.}
\label{fig:perfacc}
\end{table}

In  Figure \ref{fig:zetatiming}, we display the runtime of our algorithm for computing the sums $S_d$ as a function of $\nu$. The values were obtained on an Intel Core i7-1260P (12th Gen) 16-core processor with 32 GB of RAM. Here, the runtime only mildly depends on the choice of $\nu$, as well as of $\bm x$ and $\bm y$. Our implementation enables the computation of 2D sums in less than $20$ microseconds and 4D sums in less than \new{$0.3$} milliseconds. Even 8D sums become accessible with a runtime of less than \new{half a second} second on a standard laptop.

\new{We verify the numerical stability of our method with respect to $\bm x$ and $\bm y$ by comparison with an arbitrary-precision implementation of our method for $\Lambda=\mathds Z^d$ with $d=1,2,3$ and $\nu\in\{-1/2,1/2,3/2\}$. We consider $\bm x=(x,\dots,x)^T$ and $\bm y=(y,\dots,y)^T$ with $0\le x,y\le 1/2$, sampled in increments of $1/5$. The corresponding errors are reported in Table~\ref{tab:stability}. Owing to the full numerical stability of our incomplete gamma function implementation described in Sec.~\ref{sec:ogammaalg}, we retain full precision even in the vicinity of the singularities at $\bm x=0$ and $\bm y=0$.}

\new{Finally, Table~\ref{tab:percentages} summarizes both the number of function calls and the runtime percentages for the complete summand function and for the incomplete gamma function alone, for the benchmark functions defined in Table~\ref{tab:defanalyticsums}. For the summand function, we include all linear algebra operations as well as the evaluation of all special functions required to compute a single summand in the Crandall representation. Since an asymptotic expansion of the Crandall function is employed in the outer summation shells, the incomplete gamma function is called less often than the full summand function. While the runtime percentages of the summand function and the incomplete gamma function are nearly identical for $d=1$, the runtime share of the incomplete gamma function decreases substantially for larger system dimensions, as most evaluation points then lie in the outer summation shells.}

\section{Application to quantum systems}
\label{sec:application}

The Epstein zeta function forms the basis for studying long-range interacting classical and quantum systems with applications ranging from the computation of Madelung constants in theoretical chemistry \cite{schwerdtfeger100YearsLennardJones2024,burrowsMadelungConstantDimensions2022}, over unconventional superconductivity \cite{buchheit2023exact}, to high energy physics \cite{wolf1999exact,ambjornPropertiesVacuumMechanical1983}. In this section, we use our efficiently implemented Epstein zeta function to explore two relevant systems arising in condensed matter physics and quantum field theory.    

\subsection{Anomalous quantum spin-wave dispersion in 3D}
\label{sec:dispersion_relation}

In our first application, we study linear spin waves in a quantum spin lattice $\Lambda$ with power-law long-range interactions. The properties of the system are then determined by the Hamiltonian
\[
H=-\frac{J}{2}\sums_{\bm x,\bm y\in\Lambda}\frac{\bm S_{\bm x}\cdot \bm S_{\bm y}}{|\bm x-\bm y|^{\nu}},
\]
where
 \(\bm S_{\bm x}\) is the spin operater at lattice site \(\bm x\)  and \(J>0\) determines the ferromagnetic interaction between nearest neighbors.
 Following standard techniques (in particular Holstein--Primakoff transformation and limit of large spin quantum numbers \(S\gg 1\)), a plane wave Ansatz
yields the dispersion relation, that relates the temporal frequency $\omega(\bm k)$ of a plane wave in the system to its spatial period determined through the wavevector $\bm k$. For general lattices $\Lambda$ and interaction exponents $\nu$, this dispersion relation takes the form \cite[Appendix D]{buchheit2023exact}
\[
\hbar \omega(\bm k)=
JS\Big(\new{Z_{\Lambda,\nu}(\bm 0, \bm 0)}
-
\new{Z_{\Lambda,\nu}(\bm 0, \bm k)}\Big),
\]
with $\hbar$ the reduced Planck constant.

In 
Figure \ref{fig:dispersion}, we display the dispersion relation for a 3D square lattice $\Lambda=\mathds Z^3$ as a function of $k_1$ along the cut $k_2=k_3=0$ for different values of the interaction exponent $\nu$. Of central interest here is the behavior at small wavevectors, corresponding to the low-energy sector. While for $\nu = 5$ (green), the dispersion relation follows the well-known $\bm k^2$ behavior of short-range interacting systems, anomalous behavior is observed for $\nu<d+2$. In the case of $\nu = d+1$ (orange line), a linear dispersion relation is obtained, which is the three-dimensional equivalent of the result obtained in \cite{peter2012anomalous} for dipole-dipole interactions \(1/|\bm r|^3\) in two dimensions and generalises the result in one dimension from \cite{yusuf2004spin}.
For $\nu = 3.5$ (blue), a square root behavior around $\bm k = \bm 0$ is observed.

\begin{figure}
    \centering
        \includegraphics[width=0.53\linewidth]{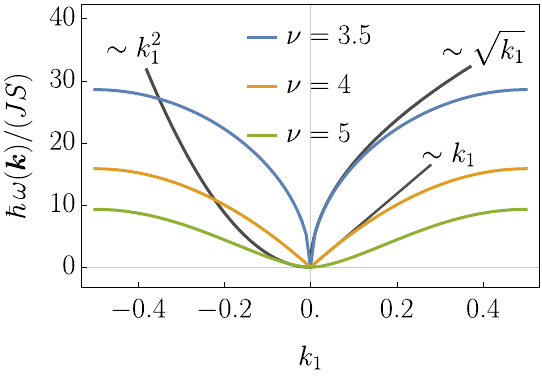}
    \caption{Quantum spin wave dispersion relation of spins on a 3D square lattice $\Lambda=\mathds Z^3$ as a function of $k_1$ for $k_2=k_3=0$.
    Typical scaling  \(\omega(\bm k)\sim \bm k^2\)
    is observed for \(\nu \ge d+2\) (green line $\nu = 5$), whereas anomalous scaling is observed for $d<\nu<d+2$  (orange line $\nu =4$, blue line $\nu=3+1/2$).}
    \label{fig:dispersion}
\end{figure}

This result can be understood from the regularised Epstein zeta function in 
Definition \ref{def:epstienreg}, which yields for the small $\vert \bm k\vert$ behavior
\[
\hbar \omega(\bm k) = JS (c_\nu \vert  \bm k\vert^{\nu-d} +\mathcal O(\bm k^2) ),
\]
with a constant $c_\nu \in \mathds R$. Hence, for \(d<\nu <d+2\), the singularity of Epstein zeta dominates the low energy behavior, leading to an anomalous scaling \(\omega(\bm k)\sim |\bm k|^{\nu-d}\). For $\nu = d$, the dispersion relation becomes unbounded from below as 
\[
\zeps{\Lambda,\nu}{\bm 0}{\bm k} \sim \log(\vert \bm k \vert^2)
\]
signaling a breakdown of the model, which is often associated with a phase transition. 

Our work enables the study of long-range dispersion relations both numerically, using our algorithm implemented in \cite{epsteinlib}, and analytically, using our analysis of the analytic properties of the Epstein zeta function and its regularisation for any lattice and any interaction exponent.

\subsection{Casimir effect}
\label{sec:casimir}

The Casimir effect states that an attractive force acts between two perfectly conducting plates due to quantum fluctuations of the electromagnetic field. Originally derived by Casimir and Polder in 1948, it can be interpreted either as a result of retarded Van-der-Waals forces \cite{Casimir:1948dh} or due to pressure associated with the spontaneous creation and annihilation of photons in vacuum \cite{casimir1948influence}. The resulting forces can dominate the physics at microscopic scales, are measurable in state-of-the-art experiments \cite{rodriguez2011casimir}, and recently have been demonstrated to be tunable from attractive to repulsive in ferrofluids \cite{zhang2024magnetic}, with various technological applications such as quantum levitation \cite{zhao2019stable}.
Our implementation of the Epstein zeta function in \cite{epsteinlib} allows for numerical investigation of the Casimir energy for a wide range of geometries \(\Lambda\) in any dimension; we present two examples.

\begin{figure}[ht] 
    \centering
        \includegraphics[width=1\linewidth]{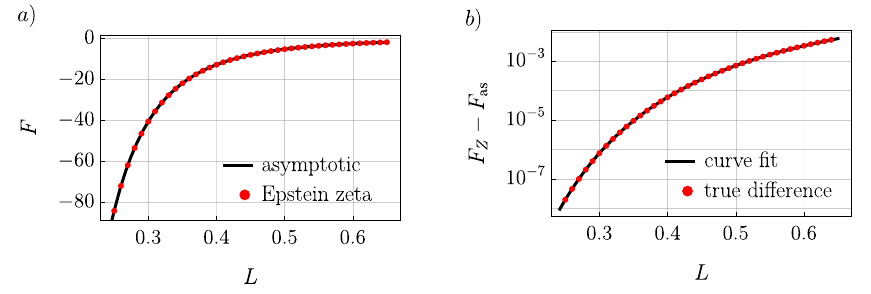}
    \caption{Comparison of the asymptotic expression of the force $F=F_{\rm as}$ of the Casimir effect and the force as the discrete derivative of the Epstein zeta function $F=F_{Z}$ as a function of 
    of the small edge
    \(L\) (a). 
    The difference between the asymptotic and the Epstein zeta function is small only for small edges,
    and the true force is bigger than the asymptotic expression for long edges (b).}
    \label{fig:vasimircutoff}
\end{figure}

\new{Consider the \(d\)-dimensional cuboid \(\Omega=\prod_{i=1}^d (0,L_i)\) with edge lengths $L_i>0$ and the associated
 lattice matrix \(A=\diag(L_1,\ldots,L_d)\).}
The Klein-Gordon wave equation
for a non-interacting scalar field with mass \(m\ge 0\), natural units $c=\hbar=1$
and periodic boundary conditions reads
\[
\begin{aligned}
&(\partial_t^2-\Delta+m^2)\phi(t,\bm x)=0, &&\qquad t\in\mathds R,\ \bm x\in \new{\Omega}\\ 
&\quad\phi(t,\bm x)=\phi(t,\bm x +\bm L_{\bm x}), &&\qquad t\in\mathds R,\ \bm x\in\partial \new{\Omega}
\end{aligned}
\]
where 
\((\bm L_{\bm x})_i=-L_i\sgn x_i\)
for \(1\le i\le d\).
A solution is given in terms of the field modes 
\[
\phi(t,\bm x)=e^{-i\lambda_{\bm k}t}
e^{\text i \bm k \cdot \bm x}
\]
with eigenvalues
\[
\lambda_{\bm k}=\sqrt{m^2+\bm k^2}
\]
where \(\bm k
=2\pi(z_1/L_1,\ldots,z_d/L_d)^T\) for \(\bm z\in\mathds Z^d\).
The total energy \(\mathcal E\) is given by 
\[
\mathcal E(\Lambda)=
\frac 12\sum_{\bm k}\lambda_{\bm k}^{-\nu}\Big|_{\nu=-1}
\]
in the sense of evaluating the meromorphic continuation in \(\nu\in\mathds C\) of the lattice sum at \(\nu = -1\), see \cite{liAttractiveRepulsiveNature1997,ederyMultidimensionalCutoffTechnique2005,ambjornPropertiesVacuumMechanical1983}. \new{We briefly explain why the meromorphic continuation yields the correct contribution to the energy density and hence to the Casimir force. Although the lattice sum is formally divergent at $\nu=-1$, the difference between the sum and the corresponding integral (representing free space without boundaries), after regularization by a $C^\infty$ cutoff function, converges to the meromorphic continuation in $\nu$ as the regularization is smoothly removed, see \cite[Thm.~5.6]{buchheit2022singular}.}
In the massless case \(m=0\) we obtain a representation in terms of the Epstein zeta function as
\[\mathcal E(\Lambda)=\pi\,\zeps{\Lambda^{\ast},-1}{\bm 0}{\bm 0}\]
where $\Lambda^*=\diag(1/L_1,\ldots,1/L_d)$.

Consider the three-dimensional box with one short edge 
\(L_1=L\ll 1\) and two edges with unit length $L_2=L_3=1$.
The Casimir effect states that there is a nonzero attractive force
\[F=-\frac{\mathrm{d}\mathcal E}{\mathrm{d} L}<0\]
between the sides of the cube with unit length.
Asymptotically, it is known
(see \cite{zeidlerQuantumFieldTheory2006}), that
\[
F\approx -\frac{\pi^2}{30 L^4}
,\qquad
L\ll 1
\]
which agrees with our numerics for small edge lengths.
We further numerically determine the difference between the force in the sense of finite difference of the Epstein zeta function $F_z$ and the asymptotic expression for the force $F_{\rm as}$ as defined above approximately equal to be
\[
F_Z-F_{\rm as}\approx -\frac{8 \pi e^{-2\pi /L}}{L^3}
\]
by a curve fit, see 
Figure \ref{fig:vasimircutoff}.

The energy of a cube with unit volume \(L_1L_2L_3=1\) is minimised by making either one or two edges small, and the highest energy state is achieved by setting $L_1=L_2=L_3$, see
Figure \ref{fig:casimirbox}.
Casimir force tends to deform cubes of fixed volume \(L_1L_2L_3=1\) into either long or thin boxes, which agrees with the analysis in \cite{ambjornPropertiesVacuumMechanical1983}.

\begin{figure}
\centering
        \includegraphics[width=.85\linewidth]{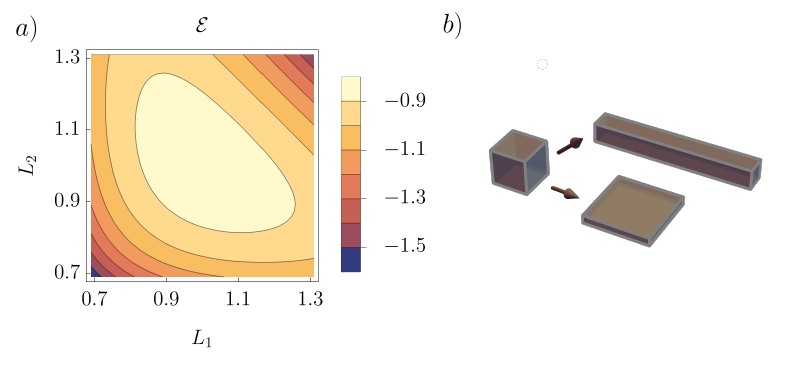}
    \caption{Total energy in a three-dimensional box with volume \(V_{\Lambda}=L_1L_2L_3=1\), as a function of two sides \(L_1\), \(L_2\), where the third side is fixed as \(L_3=1/(L_1L_2)\) (a). 
    The energy is the lowest when either two sides are small or one side is small. 
        The Casimir energy thus tends to deform cubes into either very long or very flat boxes (b).}
    \label{fig:casimirbox}
\end{figure}

In the massive case \(m>0\), the Casimir energy is no longer described by a
lattice sum as in 
Definition \ref{epsteindef}, but the derivation of Crandall's representation can still easily be adopted to yield an exponentially fast converging representation
for the total energy.
This highlights the strength of the method for the meromorphic continuation presented in this paper, which the authors will expand to even more general lattice sums in the future.

\section{Outlook and Conclusions}
\label{sec:outlook}

The Epstein zeta function is of high interest in pure and applied mathematics 
\cite{chowla1949epstein,rankinMinimumProblemEpstein1953,terras,buchheit2022singular}
with numerous applications in theoretical physics,
particularly in long-range interacting systems.
Examples include
theoretical chemistry
\cite{schwerdtfeger100YearsLennardJones2024,robles-navarroExactLatticeSummations2025},
high energy physics \cite{ambjornPropertiesVacuumMechanical1983,hawking1977zeta}
and
unconventional superconductors \cite{buchheit2023exact}.
This work
solves the long-standing issue of efficiently computing the Epstein zeta function over the whole real parameter range
and provides a comprehensive analysis of its analytic properties, making this fundamental function accessible as a powerful computational and exploratory tool.
Together with our implementation in EpsteinLib \cite{epsteinlib}, it aims to
 establish the Epstein zeta function as a standard special function.

We have derived a representation of the Epstein zeta function based on the work of Crandall, which enables accessing its meromorphic continuation and naturally yields the functional equation.
We have shown that the Epstein zeta function is jointly holomorphic in the exponent $\nu$ and in complex cones of the vector arguments $\bm x$ and $\bm y$ and have discussed its symmetries and power-law singularities.
Furthermore, we have shown that the Epstein zeta function can be decomposed into a power-law singularity
and a holomorphic function in the reciprocal elementary lattice cell, allowing for its efficient precomputation through interpolation.

Our algorithm for the efficient evaluation of the Epstein zeta function
will allow for a unifying treatment of anomalous behavior in many-body systems with long-range interactions, both classical and quantum. We have demonstrated this for quantum spin systems with long-range interactions, generalising the results from  \cite{yusuf2004spin,peter2012anomalous}. The regularised Epstein zeta function enables the efficient computation of integrals involving the Epstein zeta function, as the singularity is analytically known and can be handled via specialised techniques like Duffy transformations \cite{duffy1982quadrature}, while the regularised zeta function is 
analytic in the elementary cell of the reciprocal lattice
and can be integrated using standard Gauss quadrature, with applications to unconventional superconductivity \cite{buchheit2023exact}. 

EpsteinLib \cite{epsteinlib} is already actively being used by the quantum condensed matter community. 
The library has been applied to determine the quantum phase diagram of hardcore bosons with long-range interactions \cite{koziol2023quantum}.
It has enabled the recent study of the melting of a devil's staircase in the long-range Dicke-Ising model \cite{koziol2025melting}.
EpsteinLib has been used in the theoretical prediction and experimental verification of exotic fractional magnetisation plateaus in a Shastry-Sutherland Ising model \cite{yadav2024observation}.
Further, we have been able to reduce the scaling of computational effort for evaluating many-body lattice sums in theoretical chemistry from exponential to log-linear scaling in the number of interaction partners \cite{buchheitEpsteinZetaMethod2025}.
This method forms the foundation for an ongoing rigorous investigation of the influence of many-body interactions on the stability of \new{three-dimensional} crystal lattices in theoretical chemistry, with first results in \cite{robles-navarroExactLatticeSummations2025}. \new{Finally, we have recently applied our analysis of the complex singularity structure of generalized Epstein zeta functions to develop a new algorithm for the efficient computations of magnetic interactions between $d$-dimensional solid bodies with copies on an $n<d$-dimensional grid, thereby resolving a major open problem in the simulation of micromagnetic textures \cite{buchheitZetaExpansionLongrange2025}.}

In the future, we aim to include our  recent extension of Crandall's formula to translationally non-invariant point sets
\cite{buchheitComputationLatticeSums2024} 
in EpsteinLib, as well as other generalised zeta functions, such as the many-body zeta function introduced in \cite{buchheitEpsteinZetaMethod2025}.
We will create an algorithm for stably computing derivatives of arbitrary order of the Epstein zeta function in \(\bm x\) and $\bm y$, which correspond to generalised lattice sums \cite{buchheit2022singular} that have direct applications in the singular Euler--Maclaurin expansion \cite{buchheit2023exact}.
We plan to develop specialised algorithms for high-dimensional sums as well as expansions for lattice matrices with large condition numbers. \new{Furthermore, we aim to extend our method to support multi-shot evaluations of wavevectors and displacement vectors by exploiting redundant information, and to leverage SIMD vectorization across different architectures, such as AVX-512, to further accelerate evaluation.}
Finally,
we aim to develop and use efficient numerical methods, based on the Epstein zeta function and its generalisations,
for the precise study of exotic behavior in long-range interacting many-body systems in quantum condensed matter physics and theoretical chemistry.

\section*{Acknowledgements}

We thank Torsten Keßler, Daniel Seibel, Kirill Serkh, Gary Schmiedinghoff, and Sergej Rjasanow for insightful comments and valuable discussions that helped improve this manuscript. 
We acknowledge our DevOps engineer, Jan Schmitz, for his significant contributions to the development of the library.

\section*{Declarations}

\subsection*{Funding}
\new{This work was supported by the Klaus-Tschira Stiftung under Grant No.00.025.2025.}
J.B. acknowledges the support of the Quantum Fellowship Program of the German Aerospace Center (DLR) for funding their contribution to this work.

\subsection*{Data Availability}
The authors declare that all data supporting the findings of this work are available within this article. Our open-source library EpsteinLib is publicly available \cite{epsteinlib}.
A notebook, which reproduces all figures presented in this work, is provided in the \texttt{examples/mathematica/} folder.

\subsection*{Conflict of interest}
The authors declare that they have no conflict of interest.

\subsection*{Declaration of generative AI and AI-assisted technologies in the writing process}
During the preparation of this work, the authors used ChatGPT and Claude AI for grammar checking and minor language enhancements. After using this tool, the authors reviewed and edited the content as needed and take full responsibility for the content of the published article.


\appendix

\section{Truncation}
\label{sec:appendix-trunc}

For the investigation of the summation cutoff, it is sufficient to consider the lattice $\mathds{Z}^d$, as the following lemma shows.

\begin{lemma}
\label{lem:z-is-enough-revisited}
Let $\Lambda=A\mathds Z^d$ for $A\in\mathds R^{d\times d}$ regular.
Then, for any set $M\subset\mathds R^d\setminus\{\bm 0\}$,
and $\bm v\in\mathds R^d$, we have
\[\sum_{\bm z \in \Lambda \cap M} G_\nu(c(\bm z+\bm v)) \leq \sum_{\bm z \in \mathds{Z}^d \cap A^{-1}M} G_\nu\left(\frac{c(\bm z +A^{-1}\bm v)}{\Vert A^{-1}\Vert} \right),\]
with $\|\cdot\|$ the spectral norm.
\end{lemma}

\begin{proof}
The result follows directly from
\[|A(\bm z + A^{-1} \bm v)|= \Vert A^{-1}\Vert^{-1}\big(\Vert A^{-1}\Vert\Vert A (\bm z+A^{-1}\bm v)\Vert\big) \ge \Vert A^{-1}\Vert^{-1} |\bm z+A^{-1}\bm v|\]
as $G_\nu(\bm z)$ is rotationally symmetric and strictly decreasing in the norm of $\bm z$, see Lemma \ref{lemma:g-mon}.
\end{proof}

The following lemma allows us to bound the number of summands in shells with width $\varepsilon>0$ in the multidimensional sums in Crandall's representation.

\begin{lemma}
\label{lem:shells}
Let $d\in\mathds N_+$, let $r>0$, and let
$$
N_{d}(r)=\sup_{\bm v\in \mathds R^d}\big(\# 
\{\bm z\in\mathds Z^d: |\bm z+\bm v|\le r\}\big),
$$
where the cardinality $\#$ denotes the number of elements in a set. 
Then, for any $r>\sqrt{d}$ and $0<\varepsilon\le (r-\sqrt{d})/2$, it holds that
$$
N_d(r+\varepsilon) 
\le 
\frac{(3/2)^d\pi^{d/2}}{\Gamma(d/2+1)}r^d.
$$
\end{lemma}

\begin{proof}
Consider the disjoint union of hypercubes
$$
V=
\bigcup_{
\substack{\bm z\in\mathds Z^d+\bm v\\ |\bm z|<r+\varepsilon}}
\!\!
(\bm z+(-1/2,1/2)^d).
$$
Then $N_d(r+\varepsilon)=\Vol V$.
Since the maximum distance of any $\bm z\in\mathds R^d$ to the nearest point in $\mathds Z^d+\bm v$ is $\sqrt{d}/2$, we obtain
$
V\subset 
\overline{B}_{\tilde r}(\bm 0).
$
with $\tilde r = r+\varepsilon+\sqrt{d}/2$, from which the bound
$$
N_d(r+\varepsilon)\le \Vol  \big(\overline{B}_{\tilde r}(\bm 0)\big)= 
\pi^{d/2}{\tilde r}^d/\Gamma(d/2+1)
$$
follows.
The desired estimate is then obtained from the inequality $\varepsilon+\sqrt{d}/2\le r/2$, as then $\tilde r\le 3r/2$.
\end{proof}

For real exponents $\nu$, the properties of the upper Crandall function are a direct consequence of the integral representation.

\begin{lemma}
\label{lemma:g-mon}
Let $\nu \in \mathds{R}$, $\bm z \in \mathds{R}^d\setminus\{\bm 0\}$. Then, the upper Crandall function \(G_{\nu}(\cdot)\) is rotationally symmetric and globally decreasing in the norm of $\bm z$.
\end{lemma}

\begin{proof}
Write $G_\nu(\bm z)$ as 
\[G_{\nu}(\bm z)=\int_{-1}^1|t|^{-\nu}e^{-\pi \bm z^2/t^2}\,\frac{\rm dt}{t}\]
and note that the integrand is rotationally symmetric, positive, and strictly decreasing in \(\bm z^2\). The integral  inherits these properties.
\end{proof}

The radial symmetry of $G$ allows us to rewrite the integral in polar coordinates, reducing it to the one-dimensional case.  The following lemma enables their evaluation in terms of Crandall functions.

\begin{lemma}
\label{lem:g-int-revisited}
Let $t \in \mathds{R}$ and let $s \in \mathds{R} \setminus \{t\}$. Then, for any $r > 0$, we have
\[\int_r^\infty u^{t}G_s(u)\frac{\d u}{u} = r^{t} \frac{G_{t}(r) - G_s(r)}{t-s}.\]
\end{lemma}

\begin{proof}
After inserting the definition and substituting $\tilde u = \pi u^2$, we find
$$\int_r^\infty u^tG_s(u) \frac{\d u}{u}
=
\frac{1}{2 \sqrt{\pi}^t} \int_{\pi r^2}^\infty \tilde u^{(t-s)/2} \Gamma(s/2, \tilde u)  \frac{\d \tilde u}{\tilde u}.$$
Notice that for any $a,v\in \mathds R$, integration by parts yields
\[
v\int_{r_0}^\infty u^{v-1} \Gamma(a, u) \d u = \Gamma(a+v, r_0) - r_0^{v} \Gamma(a, r_0),\]
for any $r_0>0$,
where the boundary term vanishes due to
the superpolynomial decay of $\Gamma(a,u)$ in $u$.
We thus obtain
\[\int_r^\infty u^tG_s(u) \frac{\d u}{u} = \frac{\Gamma(t/2, \pi r^2) - (\pi r^2)^{(t-s)/2} \Gamma(s/2, \pi r^2)}{\sqrt{\pi}^t (t-s)}.\]
Rewriting in terms of the upper Crandall function completes the proof.
\end{proof}
We present the proof for the rigorous error bound in Crandall's representation. 
\begin{proof}[Proof of Theorem \ref{theorem:trunc}]
By 
Lemma \ref{lem:z-is-enough-revisited} the absolute value of the remainder is bounded by
$$
    \frac{\pi^{\nu/2}}{|\Gamma(\nu/2)|}\Bigg[\sum_{\substack{\bm z\in\mathds Z^d\\ |\bm z - A^{-1}\bm x|> r/\|A\|}}\!\!
    G_{\nu}\Big(\frac{\bm{z} -A^{-1}\bm x}{\|A^{-1}\|}\Big)
+\frac{1}{V_{\Lambda}}
\!\!\!
\sum_{\substack{\bm k\in\mathds Z^d\\ |\bm k + A^T\bm y|> r/\|A^{-1}\|}}
\!\!\!\!\!
G_{d-\nu}
\Big(\frac{\bm{k} + A^T\bm y}{\|A\|}\Big) 
\Bigg]
$$
where we enlarged the summation range using $A^{-1}(\mathds R^d\setminus B_r)\subseteq \mathds R^d\setminus B_{r/\Vert A\Vert}$ and
$A^{T}(\mathds R^d\setminus B_r)\subseteq \mathds R^d\setminus B_{r/\Vert A^{-T}\Vert}$ with $\Vert A^{-T}\Vert=\Vert A^{-1}\Vert$.

We now subdivide the sums into sums over spherical shells of width 
$\varepsilon'>0$ as follows
$$\sum_{\substack{\bm k\in \mathds Z^d\\ |\bm k + \bm v|> r'}}G_{s}
(c(\bm{k} + \bm v)) = \sum_{n = 0}^\infty \sum_{\substack{\bm k\in \mathds Z^d\\ |\bm k + \bm v|> r'+n\varepsilon'\\|\bm k + \bm v|\le  r'+(n+1)\varepsilon' }} G_{s}
(c(\bm{k} + \bm v)) 
$$
for
$r',c>0$, $\bm v\in\mathds R^d$, and $s\in \mathds C$.
We then use the strict monotonic decrease and rotational symmetry of $G_s(\cdot)$ to bound the value of $G_s$ by its maximum within the shell. We further bound the number of lattice points within the shell by the number of points within the closed ball of radius $r'+(n+1)\varepsilon'$, with the bound provided in 
Lemma \ref{lem:shells}. We obtain for 
$0<\varepsilon'\le (r'-\sqrt{d})/2$,
$$
\sum_{\substack{\bm k\in \mathds Z^d\\ |\bm k + \bm v|> r'}}G_{s}
(c(\bm{k} + \bm v))
\le
C\sum_{n=0}^{\infty}
(r'+n\varepsilon')^d
G_{s}
(c(r'+n\varepsilon')),
$$
with $C=(3/2)^d\pi^{d/2}/\Gamma(d/2+1)$.
By monotonicity,
this expression is bounded by the associated integral
$$
C\int_{-1}^{\infty}
(r'+n\varepsilon')^d
G_{s}
(c(r'+n\varepsilon'))\,\mathrm d n.
$$
Applying 
Lemma \ref{lem:g-int-revisited}, we find
$$
C
\frac 1{\varepsilon'}\Big[(r'-\varepsilon')^{d+1}
\frac{G_{d+1}(c(r'-\varepsilon')) - G_s(c(r'-\varepsilon'))}{d+1-s}
\Big]
$$
where the limit $s\to d+1$  is well-defined, since $G_{s}(u)$, $u>0$,
is holomorphic in $s\in\mathds C$ by 
Lemma \ref{lem:propCrandall}.
Letting
 $r'=r/\|A\|$, $c=1/\|A^{-1}\|$
 for the first sum, and $r'=r/\|A^{-1}\|$, $c=1/\|A\|$
 as well as $\tilde\varepsilon=c \,\varepsilon'$
 leads to the bound
$$
\kappa(A)^{d+1}C
\frac 1{\tilde\varepsilon}\Big[
\Big(\frac r{\kappa(A)}\Big)^{d+1}
\frac{G_{d+1}(r/\kappa(A)-\tilde\varepsilon) - G_s(r/\kappa(A)-\tilde\varepsilon)}{d+1-s}
\Big]
$$
where we used that $c\le 1$ since both $\|A\|\ge 1$ and $\|A^{-1}\|\ge 1$ due to $\det(A)=1$.
The bound above holds
for every
$$
0<\tilde\varepsilon\le 
(r/\kappa(A)-\sqrt{d}/\max\{\|A\|,\|A^{-1}\|\})/2.
$$
The right-hand side of the expression above is, by the constraint on $r$, larger than $\varepsilon$.
Thus, we may choose $\tilde{\varepsilon}=\varepsilon$.
Including the $\nu$-dependent prefactor yields the statement.
\end{proof}

\printbibliography

\end{document}